\RequirePackage[hyphens]{url}
\documentclass[leqno,11pt,a4paper]{amsart}
\usepackage[full]{textcomp}
\usepackage{newpxtext}
\usepackage{fontawesome}
\usepackage{cabin} % sans serif
\usepackage[varqu,varl]{inconsolata} % sans serif typewriter
\usepackage[bigdelims,vvarbb]{newpxmath} % bb from STIX
\usepackage[cal=cm,bb=esstix,bbscaled=1.05]{mathalfa} % mathcal
\usepackage[margin=2.4cm]{geometry}
\usepackage{savesym}
\let\savedegree\bigtimes
\usepackage{comment}
\let\bigtimes\relax
\usepackage{mathabx}
\usepackage{amsmath}
\let\bigtimes\savedegree
\usepackage[usenames,dvipsnames]{color}
\usepackage{hyperref}
\hypersetup{
 colorlinks,
 linkcolor={blue!90!black},
 citecolor={red!80!black},
 urlcolor={blue!50!black}
}
\usepackage{tikz}
\usepackage{tikz-cd}
\usepackage{tikz-3dplot} 

\tdplotsetmaincoords{70}{120} % set viewpoint 
\tdplotsetrotatedcoords{0}{0}{0} %<- rotate around (z,y,z)

\usepackage{graphicx}
\usepackage[all,cmtip]{xy}
\usepackage{mleftright}
\usepackage{booktabs}
\usepackage{cite}
\usepackage{enumitem}
\usepackage{amsfonts}

\setlist[enumerate]{labelsep=*, leftmargin=1.5pc}
\setlist[enumerate]{label=\normalfont(\roman*), ref=\roman*}
\newtheorem{thm}{Theorem}[section]
\newtheorem{lemma}[thm]{Lemma}
\newtheorem{cor}[thm]{Corollary}
\newtheorem{prop}[thm]{Proposition}
\theoremstyle{definition}
\newtheorem{example}[thm]{Example}

\newtheorem{remark}[thm]{Remark}
\newtheorem{definition}[thm]{Definition}
\newtheorem{definitions}[thm]{Definitions}

\newtheorem{question}[thm]{Question}
\numberwithin{equation}{section}

\newtheorem{conj}{Conjecture}

\newcommand{\Hilb}{\mathrm{Hilb}}
\newcommand{\Card}{\mathrm{Card}}
\newcommand{\Hom}{\mathrm{Hom}}
\newcommand{\Spec}{\mathrm{Spec}}

\newcommand{\Ext}{\mathrm{Ext}}
\newcommand{\gin}{\mathrm{gin}}
\newcommand{\GL}{\mathrm{GL}}
\newcommand{\soc}{\mathrm{soc}}
\newcommand{\Span}{\mathrm{Span}}
\newcommand{\ann}{\mathrm{ann}}

\renewcommand{\P}{\mathbf{P}}
\renewcommand{\AA}{\mathbf{A}}
\newcommand{\N}{\mathbf{N}}

\newcommand{\Z}{\mathbf{Z}}
\newcommand{\kk}{{\mathbf{k}}}

\newcommand{\mm}{\mathfrak{m}}
\newcommand{\lcm}{\mathrm{lcm}}

\newcommand{\ppn}{\mathrm{ppn}}
\newcommand{\pnp}{\mathrm{pnp}}
\newcommand{\npp}{\mathrm{npp}}
\newcommand{\pnn}{\mathrm{pnn}}
\newcommand{\npn}{\mathrm{npn}}
\newcommand{\nnp}{\mathrm{nnp}}

\newcommand{\p}{\mathrm{p}}
\newcommand{\n}{\mathrm{n}}

\newcommand{\bfx}{{\bf x}}

\makeatletter
\@namedef{subjclassname@2020}{%
  \textup{2020} Mathematics Subject Classification}
\makeatother

\graphicspath{{images/}}
\begin{document}
%-------------------------------------------------------------------------------
\author[R.\,Ramkumar, A. \, Sammartano]{Ritvik~Ramkumar and Alessio~Sammartano}
\address{(Ritvik Ramkumar) Department of Mathematics\\University of California at Berkeley\\Berkeley, CA\\94720\\USA}
\email{ritvik@math.berkeley.edu}
\address[Alessio Sammartano]{Dipartimento di Matematica, Politecnico di Milano, Via Bonardi 9, 20133, Milano, Italy}
\email{alessio.sammartano@polimi.it}
%-------------------------------------------------------------------------------
%-------------------------------------------------------------------------------
\title{On the tangent space to the Hilbert scheme of points in $\P^3$ }
%-------------------------------------------------------------------------------
\begin{abstract}
In this paper we study the tangent space to the Hilbert scheme $\Hilb^d \P^3$, motivated  by Haiman's work on $\Hilb^d \P^2$ and by a long-standing conjecture of  Brian\c{c}on and Iarrobino on the most singular point in $\Hilb^d \P^n$.
For points  parametrizing monomial subschemes,
we consider  a decomposition of the tangent space 
into six distinguished subspaces,
and show that a fat point  exhibits an extremal behavior in this respect.
This decomposition is also used to characterize smooth monomial points on the Hilbert scheme.
We prove the first Brian\c{c}on-Iarrobino conjecture up to a factor of $\frac{4}{3}$, and 
 improve the known asymptotic bound on the dimension of $\Hilb^d \P^3$.
 Furthermore, we construct infinitely many counterexamples to the second Brian\c{c}on-Iarrobino conjecture, and we also settle a weaker conjecture of Sturmfels in the negative.
 
\end{abstract}

\maketitle

\section*{Introduction}

The Hilbert scheme of $d$ points in $\P^n$, denoted by $\Hilb^d \P^n$, parameterizing closed zero-dimensional subschemes of $\P^n$ of degree $d$, 
is a projective moduli space with very rich geometry and a plethora of open questions. 
It was constructed by Grothendieck \cite{Grothendieck} and shown to be connected by Hartshorne \cite{Hartshorne}.
In the case of $\P^2$, Fogarty \cite{Fogarty} proved that $\Hilb^d \P^2$ is non-singular of dimension $2d$, Ellingsurd and Str\o mme \cite{EllingsrudStromme} computed its homology, 
and Arcara, Bertram, Coskun, Huizenga \cite{ABCH} studied its birational geometry in great detail. 
It also has connections to other areas of mathematics, e.g. 
to algebraic combinatorics, where it plays a central role in Haiman's proof of the $n!$ Conjecture \cite{Haiman01}. 
By contrast, the Hilbert scheme is singular for $n \ge 3$ and very little is known about its geometry. 
The case of $\Hilb^d \P^3$ is of particular interest, 
since it lies at the boundary between the smooth cases $n\leq 2$ and the cases $n\geq 4$ which are believed to be wildly pathological  \cite{Jelisiejew}.
In fact,  $\Hilb^d \mathbf{P}^3$ is known to be rather special, as it admits a super-potential description -- it is the singular locus of a hypersurface on a smooth variety, cf. \cite{BBS}.
For   $d\leq 11$, $\Hilb^d \P^3$ is irreducible
\cite{DJNT}, 
and its general point parametrizes configurations of $d$  points in $\P^3$;
in particular, the Hilbert scheme has dimension $3d$.
However,
Iarrobino \cite{Iarrobino,Iarrobino2} proved that $\Hilb^d \P^3$  is reducible for $d \geq 78$.
In general, the  dimension of $\Hilb^d \P^3$ is unknown.
Basic questions about 
dimension of tangent spaces to  $\Hilb^d \P^3$ are also wide open. 
Over forty years ago, Brian\c{c}on and Iarrobino \cite{BrianconIarrobino} established an  upper bound for the dimension of $\Hilb^d \P^n$, 
and stated two conjectures regarding the largest possible dimension of its tangent spaces.

For an ideal $I$, denote by $T(I)$  the tangent space to the corresponding point  $[I]$ in the Hilbert scheme.
The question of finding the largest possible dimension of a tangent space to $\Hilb^d \P^n$ has been raised in many places,
including e.g.  \cite{BrianconIarrobino,Sturmfels,MillerSturmfels,aim}. 
To answer this question we  restrict to an affine open $\AA^n=\Spec \, \kk[x_1,\ldots, x_n] \subseteq \P^n$. 
It is natural to expect that  a fat point subscheme $V\big((x_1, \ldots, x_n)^r\big)\subseteq \AA^n$ yields the most singular point in its own  Hilbert scheme:

\begin{conj}[\cite{BrianconIarrobino}] \label{BIconj1} 
Let $S=\kk[x_1,\ldots, x_n]$, $\mm = (x_1,\ldots, x_n)$, and $d= {r+n-1 \choose n}$ with $r \in \N$. 
For all $[I] \in \Hilb^{d} \mathbf{A}^n$ we have
$\dim_\kk T(I) \leq \dim_\kk T(\mm^r).$
\end{conj} 

No progress on the conjecture has been made so far.
By  degeneration arguments,  one reduces Conjecture \ref{BIconj1} to  monomial ideals $I$, and in fact to Borel-fixed ideals in characteristic 0.
Inspired by Haiman's theory of $\Hilb^d\AA^2$ \cite{Haiman98},
we decompose the  tangent space $T(I)$ to a monomial ideal $I \subseteq \kk[x_1,\dots,x_n]$ 
into subspaces defined in terms of $\mathbf{Z}^n$-graded directions, as follows.

\begin{definition}\label{DefinitionSubspaces} 
A {\bf signature} is a non-constant $n$-tuple on the two-element set  $\{\text{p},\text{n}\}$, where
$$
\text{p} = \text{``positive or 0''}, \quad \text{n} = \text{``negative''}.
$$
Let $\mathfrak{S}$ denote the set of signatures, and define 
for each $\mathfrak{s} \in \mathfrak{S}$
\begin{eqnarray*}
\Z^n_{\mathfrak{s}} &=& \big\{(\alpha_1, \ldots, \alpha_n) \in \Z^n: \alpha_{i} \geq 0 \, \text{ if } \, \mathfrak{s}_i = \text{p},  \, \alpha_{i} < 0 \,\, \text{ if } \, \mathfrak{s}_i = \text{n}\big\}\\
T_\mathfrak{s}(I) &=& \bigoplus_{\alpha \in \Z^n_{\mathfrak{s}}} \big|T(I)\big|_{\alpha} \subseteq T(I)
\end{eqnarray*}
where $ |T(I)|_{\alpha}$ denotes the graded component of $T(I)$ of degree $\alpha\in\Z^n$.
\end{definition}

We  then have $T_{\text{pp}\cdots\text{p}}(I)=T_{\text{nn}\cdots\text{n}}(I)=0$,
and therefore $T(I) = \bigoplus_{\mathfrak{s} \in \mathfrak{S}} T_{\mathfrak{s}}(I)$, 
cf. Proposition \ref{PropositionDecompositionDistinguishedSubspaces}. 
Our first theorem establishes a  symmetry between components of opposite signature.

\newtheorem*{thm:ppnnnp}{Theorem \ref{TheoremPairingSignatures}}
\begin{thm:ppnnnp} For any monomial point $[I]\in \Hilb^d \AA^3$ we have
\begin{eqnarray*}
\dim_\kk  T_{\ppn}(I) = \dim_\kk  T_{\nnp}(I) + d, \\
\dim_\kk  T_{\pnp}(I) = \dim_\kk  T_{\npn}(I) + d,  \\
\dim_\kk  T_{\npp}(I) = \dim_\kk  T_{\pnn}(I) + d.
\end{eqnarray*}
\end{thm:ppnnnp}
This result may be regarded as a generalization of Haiman's combinatorial proof of the smoothness of $\Hilb^d\P^2$ \cite{Haiman98}. 
In our notation, his proof  shows   that
\begin{equation}\label{EquationHaimanA2}
\dim_\kk  T_{\text{pn}}(I) = \dim_\kk  T_{\text{np}}(I) = d
\end{equation}
 for any monomial point $[I]\in \Hilb^d \AA^2$. 
Theorem \ref{TheoremPairingSignatures} extends \eqref{EquationHaimanA2} to $\AA^3$ in the sense that it implies
 $$
 \dim_\kk T_{\mathrm{pnp}}(I) + \dim_\kk T_{\mathrm{pnn}}(I) = \dim_\kk T_{\mathrm{npp}}(I) + \dim_\kk T_{\mathrm{npn}}(I)
 $$
  and two other similar equations.
Our result may be seen as further  evidence for the exceptionality of the Hilbert scheme of points in $\mathbf{P}^3$. 
For instance, it implies that $\dim_\kk  T(I)$ has the same parity as the length $d = \dim_\kk(S/I)$, a fact established in \cite{MNOP}
where it plays a crucial role in the calculation of Donaldson-Thomas theory for toric threefolds.
We are not aware of any such symmetry phenomenon in higher dimension.

As a special case,  Theorem \ref{TheoremPairingSignatures} provides a simple criterion for smoothness of  monomial points on the Hilbert scheme, in terms of the 
 subspaces $T_\mathfrak{s}(I)$.

\newtheorem*{thm:smoothchar}{Theorem \ref{TheoremSmoothnessCriterion}}
\begin{thm:smoothchar} A monomial point $[I] \in \Hilb^d \AA^3 $ is smooth  if and only if
$$T_\mathfrak{s}(I)=0
\qquad 
\text{for all }\mathfrak{s}\in \{\p\n\n, \n\p\n, \n\n\p\}.
$$
\end{thm:smoothchar}

In the opposite direction, we use the subspaces   $T_\mathfrak{s}(I)$  to provide  evidence in favor of Conjecture \ref{BIconj1}.
Clearly, Conjecture \ref{BIconj1} is implied by the  statement that $ \dim_\kk T_{\mathfrak{s}}(I) \leq  \dim_\kk T_{\mathfrak{s}}(\mm^r)$ for all $\mathfrak{s} \in \mathfrak{S}$ and all   Borel-fixed points $[I]$. 
For $\Hilb^d\AA^3$, we are able to establish this inequality for four out of the six signatures $\mathfrak{s}$. 
As a bonus, we characterize when equality holds.

\newtheorem*{thm:ppn}{Theorem \ref{TheoremExtremalSubspaces}}
\begin{thm:ppn} Let $d={r +2 \choose 3}$ and let $[I] \in \Hilb^d \mathbf{A}^3$ be  Borel-fixed, with $\mathrm{char}\, \kk = 0$. We have
\begin{eqnarray*}
\dim_\kk T_{\ppn}(I) \leq \dim_\kk T_{\ppn}(\mm^r), & \, & \dim_\kk  T_{\nnp}(I) \leq \dim_\kk T_{\nnp}(\mm^r),  \\
\dim_\kk  T_{\pnp}(I) \leq \dim_\kk T_{\pnp}(\mm^r), & \, & \dim_\kk  T_{\npn}(I) \leq \dim_\kk T_{\npn}(\mm^r). 
\end{eqnarray*}
Moreover, in each case equality occurs  if and only if $I = \mathfrak{m}^r$.
\end{thm:ppn}
 
We conjecture that $ \dim_\kk T_{\npp}(I) \leq  \dim_\kk T_{\npp}(\mm^r)$ and $ \dim_\kk T_{\pnn}(I) \leq  \dim_\kk T_{\pnn}(\mm^r)$ as well, but we are unable to prove this. 
However, we are able to prove  Conjecture \ref{BIconj1} up to a factor of $\frac{4}{3}$.
This also allows us  to improve the  asymptotic bound on the dimension of $\Hilb^d\P^3$, 
 a problem proposed by Sturmfels in
 \cite[Problem 2.4.c]{Sturmfels}.

\newtheorem*{thm:globalbound}{Theorem \ref{TheoremGlobalBound}}
\begin{thm:globalbound} 
For all $d \in \N$ and  $[I] \in \Hilb^d \P^3$ we have
\begin{equation*}
 \dim_\kk  T(I) \leq \frac{4}{3}\dim_\kk  T(\mm^r)  \approx 3.63\cdot d^{\frac{4}{3}} + O(d)
\end{equation*}
whenever  $ d \leq \binom{r+2}{3}$. 
In particular,  $\dim \Hilb^d \P^3 \leq 3.64\cdot d^{\frac{4}{3}}$ for $ d \gg 0$.
\end{thm:globalbound}

Note that Theorem \ref{TheoremGlobalBound} also holds for Hilbert schemes of points of arbitrary smooth threefolds, since these are \'etale-locally isomorphic to $\Hilb^d \P^3$, see for instance \cite[Lemma 4.4]{BehrendFantechi}.

Finally,
we address what makes $\Hilb^d \AA^3$ interesting  in the  case   $d={r+2 \choose 3}$.
It would be desirable to find a version  of Conjecture \ref{BIconj1} that holds for arbitrary $d\in \N$.
A natural guess comes  from the  point $[E(d)]\in \Hilb^d \AA^3$, where $E(d) \subseteq S$ denotes the unique lexsegment ideal such that 
$\dim_\kk\big(S/E(d)\big) = d $ and 
$\mm^{r+1} \subseteq E(d) \subseteq \mm^r$ for some $r$.
In the same paper, Brian\c{c}on and Iarrobino formulated the following extension of Conjecture \ref{BIconj1}:

\begin{conj}[\cite{BrianconIarrobino}] \label{BIconj2} 
Let   $d\in \N$.
For all $[I] \in \Hilb^{d} \mathbf{A}^n$ we have
$\dim_\kk T(I) \leq \dim_\kk T\big(E(d)\big).$
\end{conj} 

However,  Conjecture \ref{BIconj2} turned out to be false.
Using a computer search, Sturmfels \cite{Sturmfels} disproved it for $n=3$  by exhibiting a counterexample for $d = 8$ and  one for $d =16$.
We show that these two counterexamples are not sporadic.

\newtheorem*{thm:counterex}{Theorem \ref{TheoremCounterexamples}}
\begin{thm:counterex} 
For each $r\geq 3$ there exist  $ {r+2 \choose 3} < d \leq  {r +3 \choose 3}$ 
and $[I] \in \Hilb^d\AA^3$ such that 
$$
\dim_\kk T\big(I\big) > \dim_\kk T\big(E(d)\big).
$$
\end{thm:counterex}

Attempting to salvage Conjecture \ref{BIconj2},
Sturmfels asked if, for each $d$, the largest tangent space dimension for $\Hilb^d\AA^3$
 is attained at an initial ideal of the generic configuration of $d$ points in $\AA^3$  \cite[Problem 2.4.a]{Sturmfels}.
 Such points are more special than the Borel-fixed ones, but include $E(d)$ as a particular example. 
However, we show in Proposition \ref{PropositionCounterexampleSturmfels} that Sturmfels' question has a negative answer, for $d=39$.

\medskip

\noindent{\bf Organization.}
In Section \ref{SectionTangentSpacePreliminaries} we introduce a combinatorial framework for the tangent space $T(I)$.
In Section \ref{SectionSymmetry} we explore the symmetries between the subspaces $T_\mathfrak{s}(I)$, for arbitrary monomial ideals $I$.
We prove Theorem \ref{TheoremPairingSignatures}, and,
as a corollary, we deduce the smoothness criterion of Theorem \ref{TheoremSmoothnessCriterion}.
In Section \ref{SectionExtremality} we investigate the subspaces $T_\mathfrak{s}(I)$ for Borel-fixed ideals $I$, 
and prove Theorem \ref{TheoremExtremalSubspaces}.
In Section \ref{SectionGlobalEstimates} we prove Theorem \ref{TheoremGlobalBound}, 
exploiting duality in 2-dimensional regular local rings.
We conclude the paper by  showing in Section \ref{SectionCounterexamples} that Conjecture \ref{BIconj2} fails infinitely often, and by exhibiting a counterexample  to Sturmfels' more general question.

\section{The tangent space}\label{SectionTangentSpacePreliminaries}

Let  $\kk$ denote an infinite field, $S = \kk[x_1,\ldots, x_n] $  the polynomial ring in $n$ variables, 
  $\mm= (x_1, 	\ldots, x_n)$  the ideal of the origin in $\AA^n= \Spec(S)$.
When $n\leq 3$, we typically denote the variables by $x,y,z$ instead of $x_1, x_2, x_3$.
We denote vectors by $\alpha = (\alpha_1, \ldots, \alpha_n) \in \Z^n$ and use the notation $\bfx^\alpha = x_1^{\alpha_1}\cdots x_n^{\alpha_n}$ for monomials of $S$.
If $V$ is a (multi)graded vector space,  we use the notation $|V|_{\alpha}$ to  denote the graded component  of $V$ of   degree $\alpha$.

The main object of interest is the Hilbert scheme $\Hilb^d\AA^n$ parametrizing 0-dimensional subschemes of $\AA^n$ of length $d$, 
equivalently ideals $I \subseteq S$ with $\dim_\kk(S/I) = d$.
We use the notation $[I]$ to denote the $\kk$-point in the Hilbert scheme corresponding to an ideal $I$.
The Zariski tangent space to a point $[I]\in \Hilb^d\AA^n$ may be identified with the $\kk$-vector space  \cite[18.29]{MillerSturmfels}
$$T(I) = \Hom_S(I, S/I).$$ 

We say that $[I]$ is a monomial point of the Hilbert scheme if $I$ is a monomial ideal, and likewise for other attributes of ideals, such as Borel-fixed or strongly stable.
The group $\GL_n$ acts on $S$ and $\Hilb^d\AA^n$ by change of coordinates.
The monomial points are exactly those fixed by  the torus $(\kk^*)^n \subseteq \GL_n$.
A point is {\bf Borel-fixed} if it is fixed by the Borel subgroup of $\GL_n$ consisting of upper triangular matrices;
note that Borel-fixed implies monomial.
The well-known generic initial ideal deformation allows to reduce questions such as Conjectures \ref{BIconj1} and \ref{BIconj2} to the case of Borel-fixed  points,
see \cite[15.9]{Eisenbud} or \cite[2.2--2.3]{MillerSturmfels} for details.

\begin{lemma}\label{LemmaGinReduction}
For every $[I] \in \Hilb^d\AA^n$ we have $\dim_\kk  T(I) \leq \dim_\kk  T(\gin\, I)$.
Moreover, $\gin\, I \subseteq S$ is Borel-fixed.
\end{lemma}

In characteristic 0, Borel-fixed ideals are characterized by an  exchange property.

\begin{lemma} \label{LemmaBorelExchangeChar0} 
Assume $\text{char} \, \kk =0$. 
A monomial ideal $I \subseteq S$ is Borel-fixed if and only if  it is {\bf strongly stable}, that is, for any monomial $\bfx^\alpha\in I$ with   $\alpha_j>0$ we have $\frac{x_i}{x_j}\bfx^\alpha \in I$ for all $i < j$.
\end{lemma}

For a monomial point $[I] \in \Hilb^d\AA^n$ the tangent space $T(I)$ inherits  a natural $\Z^n$-grading. 
Our next goal is to describe a combinatorial interpretation of $T(I)$  in terms of regions in $\Z^n$.

\begin{definitions}
For a monomial ideal $I$, we define $\tilde{I} \subseteq \N^n$ to be the subset consisting of the exponent vectors of all monomials in $I$.

A \textbf{path} between  $\alpha,\beta \in \Z^n$ is a sequence 
$\alpha=\gamma^{(0)}, \gamma^{(1)},\ldots,\gamma^{(m-1)}, \gamma^{(m)}=\beta$ of points of $\Z^n$ 
such that  $ \Vert \gamma^{(i+1)}-\gamma^{(i)} \Vert =1$ for all $i$,
where  $\Vert \delta\Vert=\sum_{j=1}^n|\delta_j|$ denotes the $1-$norm in $\Z^n$ .

A subset $U \subseteq \Z^n$ is said to be \textbf{connected} if it is non-empty and for any two points  $\alpha,\beta \in U$ there is a path between them contained in $U$.
Given a subset $V \subseteq \Z^n$, a maximal connected subset $U\subseteq V$ is called a \textbf{connected component}.

A  subset $U\subseteq \Z^n$ is  \textbf{bounded} if $\Card(U) <\infty.$
\end{definitions}

\begin{remark}
Let $[I] \in \Hilb^d\AA^n$ and $\alpha\in \Z^n$.
A  connected component $U$ of $ (\tilde{I}+{\alpha}) \setminus \, \tilde{I}$ is  bounded if and only if $ U \subseteq \N^n$.
The condition is sufficient  as $\Card(\N^n\setminus 	\tilde{I}) <\infty$, 
and necessary since if $\beta \in U $ with $\beta_i<0$,
then $\beta+m\mathbf{e}_j \in U$ for all $m\in \N$ and $j \ne i$, where $\mathbf{e}_j\in\N^{n}$ is the $j$-th basis vector.
\end{remark}

\begin{prop} \label{PropositionBasis} Let $\alpha \in \Z^n$ and $[I]\in \Hilb^d\AA^n$. 
The set of bounded connected components of $(\tilde{I}+{\alpha}) \setminus \, \tilde{I}$  corresponds  to a basis of $|T(I)|_\alpha$.
\end{prop}
\begin{proof} 
For each 
 bounded connected component  $U \subseteq (\tilde{I}+{\alpha}) \setminus \, \tilde{I}$ 
we define a multigraded $\kk-$linear map $\varphi_U: I \rightarrow S/I$ by setting $\varphi_U(\bfx^\beta) = \bfx^{\alpha+\beta}\in S/I$ if ${\alpha+\beta}\in U$, 0 otherwise.
We claim that $\varphi_U$ is $S$-linear;
it suffices to check the equation   $\phi(\bfx^\beta \bfx^\gamma) = \bfx^\beta\phi(\bfx^\gamma)$ in $S/I$ for all $\beta \in \N^n, \gamma \in \tilde{I}$.
This is clearly true if $\alpha+\beta+\gamma \in \tilde{I}$.
If $\alpha+\beta+\gamma \notin \tilde{I}$, observe that $\alpha+\beta+\gamma \in U$ if and only if $\alpha+\gamma \in U$, since the two points are connected in 
$(\tilde{I}+{\alpha}) \setminus \, \tilde{I}$, thus the equation holds and $\varphi_U \in |T(I)|_\alpha$.
We have $\mathrm{Image}(\varphi_U)= \Span_\kk(\bfx^\alpha \, : \, \alpha \in U) \subseteq S/I$, hence all  maps $\varphi_U$ are linearly independent.  

Finally, let $\psi \in  |T(I)|_\alpha$ be any  map.
If $ \beta, \gamma \in \tilde{I}$ are such that $\alpha+\beta, \alpha + \gamma $ lie in the same connected component $U \subseteq (\tilde{I}+{\alpha}) \setminus \, \tilde{I}$,
then there exists $c_{\psi,U} \in \kk$ such that $\psi( \bfx^\beta) = c_{\psi,U} \bfx^{\alpha+\beta}$ and $\psi( \bfx^\gamma) = c_{\psi,U} \bfx^{\alpha+\gamma}$:
this claim follows easily by induction on $\Vert \beta - \gamma \Vert$.
In particular, $c_{\psi,U}=0$ if $U$ is unbounded.
We deduce that $\psi = \sum_U c_{\psi,U} \varphi_U$, concluding the proof.
\end{proof}

\begin{example}
Let $I= (x_1^2,x_2^2,x_2x_3, x_3^2) \subseteq S = \kk[x_1,x_2,x_3]$ and $\alpha = (0,-1,0)$.
We illustrate the regions of Proposition \ref{PropositionBasis}:

\begin{center}
\begin{tikzpicture}[scale=1,tdplot_rotated_coords,
                    rotated axis/.style={->,purple,ultra thick},
                    blackBall/.style={ball color = black!80},
                    borderBall/.style={ball color = white,opacity=.25}, 
                    very thick]

\foreach \x in {0,1,2}{\foreach \z in {0,1,2}{ \foreach \y in {2}{ \filldraw[black] (\x,\y,\z) circle (0.35mm);}}}

\foreach \x in {0,1,2}{\foreach \z in {1,2}{  \filldraw[black] (\x,1,\z) circle (0.35mm);}}

\foreach \x in {0,1,2}{\foreach \z in {2}{  \filldraw[black] (\x,0,\z) circle (0.35mm);}}

\foreach \x in {2}{\foreach \z in {0,1,2}{ \foreach \y in {0,1,2}{ \filldraw[black] (\x,\y,\z) circle (0.35mm);}}}

\foreach \z in {0,1,2}{ \foreach \y in {2}{ 
\draw[thin] (0,\y,\z)--(2,\y,\z);
\draw[thin] (2,\y,\z)--(2.4,\y,\z);}}

\foreach \z in {0,1}{ \foreach \y in {0,1}{ 
\draw[thin] (2,\y,\z)--(2,\y,\z);
\draw[thin] (2,\y,\z)--(2.5,\y,\z);}}

\foreach \z in {1,2}{
\draw[thin] (0,1,\z)--(2,1,\z);
\draw[thin] (2,1,\z)--(2.5,1,\z);}

\foreach \z in {2}{
\draw[thin] (0,0,\z)--(2,0,\z);
\draw[thin] (2,0,\z)--(2.5,0,\z);}

\foreach \x in {2}{\foreach \z in {0,1,2}{  
\draw[thin] (\x,0,\z)--(\x,2,\z);
\draw[thin] (\x,2,\z)--(\x,2.4,\z);}}

\foreach \x in {0,1}{
\draw[thin] (\x,2,0)--(\x,2,0);
\draw[thin] (\x,2,0)--(\x,2.4,0);}

\foreach \x in {0,1}{
\draw[thin] (\x,1,1)--(\x,2,1);
\draw[thin] (\x,2,1)--(\x,2.4,1);}

\foreach \x in {0,1}{
\draw[thin] (\x,0,2)--(\x,2,2);
\draw[thin] (\x,2,2)--(\x,2.4,2);}

\foreach \x in {0,1}{
\draw[thin] (\x,1,1)--(\x,1,2);}

\foreach \x in {0,1,2}{\foreach \y in {2,2}{  
\draw[thin] (\x,\y,0)--(\x,\y,2);
\draw[thin] (\x,\y,2)--(\x,\y,2.4);}}

\foreach \x in {2,2}{\foreach \y in {0,1}{  
\draw[thin] (\x,\y,0)--(\x,\y,2);
\draw[thin] (\x,\y,2)--(\x,\y,2.4);}}

\foreach \x in {0,1}{\foreach \y in {0,1}{  
\draw[thin] (\x,\y,2)--(\x,\y,2);
\draw[thin] (\x,\y,2)--(\x,\y,2.4);}}

\foreach \x in {0,1}{\foreach \z in {0,1}{
\draw[thin,  dotted] (\x,0,\z)--(\x,2,\z);}}

\foreach \z in {0,1}{\foreach \y in {0,1}{
\draw[  thin,  dotted] (0,\y,\z)--(2,\y,\z);}}

\foreach \x in {0,1}{\foreach \y in {0,1}{
\draw[thin,  dotted] (\x,\y,0)--(\x,\y,2);}}

\draw[thin,->] (2,0,0)--(3.5,0,0);
\draw[thin,->] (0,2,0)--(0,3,0);
\draw[thin,->] (0,0,2)--(0,0,3);

\foreach \p in {(0,0,0),(1,0,0), (0,1,0), (1,1,0), (0,0,1), (1,0,1)}{
\draw \p circle (0.35mm);
\filldraw[white] \p circle (0.3mm);
}
\node at (0,0.5,3) {$x_3$};
\node at (0,3,0.4) {$x_2$};
\node at (3.5,-0.5,0) {$x_1$};
\node at (4,2,0) {$\tilde{I}$};

\end{tikzpicture}
\hspace{0.5cm}
\begin{tikzpicture}[scale=1,tdplot_rotated_coords,
                    rotated axis/.style={->,purple,ultra thick},
                    blackBall/.style={ball color = black!80},
                    borderBall/.style={ball color = white,opacity=.25}, 
                    very thick]

\foreach \x in {0,1,2}{\foreach \z in {0,1,2}{ \foreach \y in {2}{ \filldraw[black] (\x,\y,\z) circle (0.35mm);}}}

\foreach \x in {0,1,2}{\foreach \z in {1,2}{  \filldraw[black] (\x,1,\z) circle (0.35mm);}}

\foreach \x in {0,1,2}{\foreach \z in {2}{  \filldraw[black] (\x,0,\z) circle (0.35mm);}}

\foreach \x in {2}{\foreach \z in {0,1,2}{ \foreach \y in {0,1,2}{ \filldraw[black] (\x,\y,\z) circle (0.35mm);}}}

\foreach \z in {0,1,2}{ \foreach \y in {2}{ 
\draw[thin] (0,\y,\z)--(2,\y,\z);
\draw[thin] (2,\y,\z)--(2.4,\y,\z);}}

\foreach \z in {0,1}{ \foreach \y in {0,1}{ 
\draw[thin] (2,\y,\z)--(2,\y,\z);
\draw[thin] (2,\y,\z)--(2.5,\y,\z);}}

\foreach \z in {1,2}{
\draw[thin] (0,1,\z)--(2,1,\z);
\draw[thin] (2,1,\z)--(2.5,1,\z);}

\foreach \z in {2}{
\draw[thin] (0,0,\z)--(2,0,\z);
\draw[thin] (2,0,\z)--(2.5,0,\z);}

\foreach \x in {2}{\foreach \z in {0,1,2}{  
\draw[thin] (\x,0,\z)--(\x,2,\z);
\draw[thin] (\x,2,\z)--(\x,2.4,\z);}}

\foreach \x in {0,1}{
\draw[thin] (\x,2,0)--(\x,2,0);
\draw[thin] (\x,2,0)--(\x,2.4,0);}

\foreach \x in {0,1}{
\draw[thin] (\x,1,1)--(\x,2,1);
\draw[thin] (\x,2,1)--(\x,2.4,1);}

\foreach \x in {0,1}{
\draw[thin] (\x,0,2)--(\x,2,2);
\draw[thin] (\x,2,2)--(\x,2.4,2);}

\foreach \x in {0,1}{
\draw[thin] (\x,1,1)--(\x,1,2);}

\foreach \x in {0,1,2}{\foreach \y in {2,2}{  
\draw[thin] (\x,\y,0)--(\x,\y,2);
\draw[thin] (\x,\y,2)--(\x,\y,2.4);}}

\foreach \x in {2,2}{\foreach \y in {0,1}{  
\draw[thin] (\x,\y,0)--(\x,\y,2);
\draw[thin] (\x,\y,2)--(\x,\y,2.4);}}

\foreach \x in {0,1}{\foreach \y in {0,1}{  
\draw[thin] (\x,\y,2)--(\x,\y,2);
\draw[thin] (\x,\y,2)--(\x,\y,2.4);}}

\foreach \x in {0,1}{\foreach \z in {0,1}{
\draw[thin,  dotted] (\x,0,\z)--(\x,2,\z);}}

\foreach \z in {0,1}{\foreach \y in {0,1}{
\draw[  thin,  dotted] (0,\y,\z)--(2,\y,\z);}}

\foreach \x in {0,1}{\foreach \y in {0,1}{
\draw[thin,  dotted] (\x,\y,0)--(\x,\y,2);}}

\draw[thin,->] (2,1,0)--(3.5,1,0);
\draw[thin,->] (0,2,0)--(0,3,0);
\draw[thin,->] (0,1,2)--(0,1,3);

\foreach \p in {(0,0,0),(1,0,0), (0,1,0), (1,1,0), (0,0,1), (1,0,1)}{
\draw \p circle (0.35mm);
\filldraw[white] \p circle (0.3mm);
}

\node at (0,1.5,3) {$x_3$};
\node at (0,3,0.4) {$x_2$};
\node at (3.5,0.5,0) {$x_1$};
\node at (4,3,0) {$\tilde{I}+\alpha$};
\end{tikzpicture}
\hspace{0.5cm}
\begin{tikzpicture}[scale=1,tdplot_rotated_coords,
                    rotated axis/.style={->,purple,ultra thick},
                    blackBall/.style={ball color = black!80},
                    borderBall/.style={ball color = white,opacity=.25}, 
                    very thick]

\foreach \x in {0,1,2}{\foreach \y in {0,1,2}{
\draw[thin, dotted] (\x,\y,0)--(\x,\y,2);}}

\foreach \x in {0,1,2}{\foreach \z in {0,1,2}{
\draw[thin, dotted] (\x,0,\z)--(\x,2,\z);}}

\foreach \z in {0,1,2}{\foreach \y in {0,1,2}{
\draw[thin, dotted] (0,\y,\z)--(2,\y,\z);}}

\draw[thin,dotted,->] (0,1,0)--(3.5,1,0);
\draw[thin,dotted,->] (0,2,0)--(0,3,0);
\draw[thin,dotted,->] (0,1,0)--(0,1,3);

\foreach \x in {0,1,2}{\foreach \y in {0,1,2}{\foreach \z in {0,1,2}{
\draw (\x,\y,\z) circle (0.35mm);
\filldraw[white] (\x,\y,\z) circle (0.3mm);}}}

\foreach \p in {(0,2,0),(1,2,0),(0,1,1),(1,1,1), (0,0,2),(1,0,2),(2,0,2),(2,0,1),(2,0,0)}{
\filldraw \p circle (0.35mm);}
\draw[thin] (0,1,1)--(1,1,1);
\draw[thin] (0,2,0)--(1,2,0);
\draw[thin] (0,0,2)--(2.5,0,2);
\draw[thin] (2,0,1)--(2.5,0,1);
\draw[thin] (2,0,0)--(2.5,0,0);
\draw[thin] (2,0,0)--(2,0,2.5);
\draw[thin] (1,0,2)--(1,0,2.5);
\draw[thin] (0,0,2)--(0,0,2.5);

\node at (0,1.5,3) {$x_3$};
\node at (0,3,0.4) {$x_2$};
\node at (3.5,0.5,0) {$x_1$};
\node at (4,3,0) {$\big(\tilde{I}+\alpha\big) \setminus \tilde{I}$};
\end{tikzpicture}
\end{center}
The set $\big(\tilde{I}+\alpha\big) \setminus \tilde{I}$ contains two bounded connected components and one unbounded connected component, so $\dim_\kk |T(I)|_{\alpha}=2$.
\end{example}

\begin{remark} \label{RemarkCharacteristicZero} 
A simple but useful consequence of Proposition \ref{PropositionBasis} is the fact that, for $I$ monomial, $\dim_\kk  T(I)$ is independent of $\kk$. 
Thus, in Conjectures \ref{BIconj1}  and \ref{BIconj2} we may assume $\text{char} \, \kk =0$.
\end{remark}

\begin{remark}\label{RemarkArrows}
For $n=2$, 
the tangent space  $T(I)$ is analyzed combinatorially in \cite{Haiman98} in terms of ``arrows'', 
see also \cite[18.2]{MillerSturmfels}.
That description is essentially equivalent to the one presented here, in Proposition \ref{PropositionBasis}.
However, 
we find  the framework of connected components to be more  transparent and efficient.
\end{remark}

Recall the distinguished subspaces of $T(I)$ introduced in Definition \ref{DefinitionSubspaces}.
These are the only relevant subspaces of the tangent space:

\begin{prop} \label{PropositionDecompositionDistinguishedSubspaces} 
If $[I]\in \Hilb^d\AA^n$ is a monomial point and $n\geq 2$, then
$T(I) = \bigoplus_{\mathfrak{s} \in \mathfrak{S}} T_{\mathfrak{s}}(I).$
\end{prop}
\begin{proof}
Let $\alpha\in \Z^n$.
If $\alpha_i \geq 0 $ for all $i$, then 
 $\tilde{I}+{\alpha} \subseteq \tilde{I}$ and  $(\tilde{I}+{\alpha}) \setminus\, \tilde{I} = \emptyset$.
Suppose $\alpha_i <0 $ for all $i$,  we claim that
 $(\tilde{I}+{\alpha}) \setminus\, \tilde{I}$ is connected and unbounded.
 To see this, notice that the ``boundary'' $B = \tilde{I}\,\setminus \big( \tilde{I} + (1, 1, \ldots, 1) \big)$ is connected and unbounded.
Furthermore,  $(B + \alpha) \cap \tilde{I}= \emptyset$,
so $(B + \alpha)\subseteq (\tilde{I}+{\alpha}) \setminus\, \tilde{I} $ is connected and unbounded.
However, any point of $(\tilde{I}+{\alpha}) \setminus\, \tilde{I} $ is connected to $(B + \alpha)$,
since any point of $\tilde{I}$ is connected to $B$ by a straight path, 
and this  verifies the claim.
In either case $|T(I)|_\alpha = 0$ by Proposition~\ref{PropositionBasis}.
\end{proof}

For a monomial point  $[I] \in \Hilb^d\AA^2$
Proposition \ref{PropositionDecompositionDistinguishedSubspaces} gives  the decomposition
$$
T(I)  = T_{\p\n}(I) 	\oplus   T_{\n\p}(I), 
$$
whereas for a monomial point  $[I] \in \Hilb^d\AA^3$  we have
$$
T(I)  = T_{\ppn}(I) 	\oplus   T_{\pnp}(I) 	\oplus   T_{\npp}(I) 	\oplus   T_{\pnn}(I) 	\oplus   T_{\npn}(I) 	\oplus   T_{\nnp}(I).
$$

Next, we compute the components of  the tangent space for the fat point $[\mm^r]$.
For any vector $\alpha = (\alpha_1, \ldots, \alpha_n)\in \Z^n$
we have $\alpha = \alpha^+ - \alpha^-$ for two unique vectors $\alpha^+,\alpha^-\in \N^n$ such that $\alpha^+ \cdot \alpha^- = 0$.
Moreover  we denote  $\omega(\alpha) = \alpha_1 + \cdots + \alpha_n \in \Z$.

\begin{lemma}\label{LemmaHilbertFunctionFatPoint}
Let $\alpha\in \Z^n$ and $r \in \N$.
We have $|T(\mm^r)|_\alpha=0$ if $\omega(\alpha) \ne -1$.
If $\omega(\alpha)=  -1$ then $\dim_\kk |T(\mm^r)|_\alpha={n + r - \omega(\alpha^-) -1\choose n-1 }
$
if  $ \omega(\alpha^-) \leq r$, $|T(\mm^r)|_\alpha= 0$ otherwise.
\end{lemma}
\begin{proof}
For simplicity we denote $M = \widetilde{\mm^r} \subseteq \N^n$.
If $\omega(\alpha) \geq 0$ then$\big((M + \alpha)\, \setminus \, M\big)\cap \N^n = \emptyset$,
while 
if $\omega(\alpha) \leq -2$ 
then the whole region
$(M + \alpha)\, \setminus \, M$ is connected and unbounded, as it follows by inspecting  the points $\beta+\alpha\in (M + \alpha)$ with $\omega(\beta) = r, r+1$.
In either case  $|T(\mm^r)|_\alpha = 0 $ by Proposition \ref{PropositionBasis}.

If $\omega(\alpha)=  -1$ then  any bounded  component of $(M + \alpha)\, \setminus \, M$ consists of a single point $\beta+\alpha\in \N^n$ with $\omega(\beta) = r$.
These points are in bijection with  points $\gamma= \beta-\alpha^-\in \N^n$ such that $\omega(\gamma) = r - \omega(\alpha^-)$,
i.e. with the monomials of degree $ r - \omega(\alpha^-)$, yielding the desired formula.
\end{proof}

Finally, 
we distinguish some special tangent vectors in $T(I)$.
For an $S$-module $M$, we denote its {\bf socle} by $\soc(M) = 0 :_M \mm\subseteq M$.
Notice that
$\soc(T(I)) = \Hom_S(I, \soc(S/I)) \subseteq T(I)$.

\begin{remark}\label{RemarkBasisSocle}
If $[I]\in \Hilb^d\AA^n$ is monomial, then $\soc(S/I)$ and   $\soc(T(I))$ are $\Z^n-$graded.
Furthermore,  we see from the proof of Proposition \ref{PropositionBasis} that 
a $\kk$-basis for $\big|\soc(T(I))\big|_\alpha$ is 
given by the maps $\varphi_U$ where $U\subseteq (\tilde{I}+{\alpha}) \setminus \, \tilde{I}$ is a connected component such that $\Card(U) =1 $.
We refer to these $\varphi_U$'s as the \emph{socle maps}.
\end{remark}

It is easy to compute $\dim_\kk \soc(T(I))$, using the isomorphism 
\begin{equation}\label{EquationSocleTangentSpace}
\soc\big(T(I)\big) =  \Hom_S\left(I,\, \soc\left(\frac{S}{I}\right)\right) \cong \Hom_\kk\left(\frac{I}{\mm I},\, \soc\left(\frac{S}{I}\right)\right).
\end{equation}
When $I = \mm^r$  we have $  \soc(T(I)) = T(I)$ by Lemma \ref{LemmaHilbertFunctionFatPoint},
 but in general the inclusion is strict.

\begin{remark}
For other instances in which
deformations of
 points in the Hilbert scheme admit explicit combinatorial descriptions, 
see the works of  Altmann and Christophersen \cite{AC1,AC2} on square free monomial ideals.
\end{remark}

\section{Symmetries in the tangent space and smooth points}\label{SectionSymmetry}

In the rest of the paper  we work with the Hilbert scheme of points in $\AA^3$, 
so we fix $S = \kk[x,y,z]$ and $\mm=(x,y,z)$, unless stated otherwise.
We explore  symmetries between the components $T_\mathfrak{s}(I)$ of the tangent space introduced in Definition \ref{DefinitionSubspaces}.
The main results of this section are Proposition \ref{PropositionDualityBij} and Theorem \ref{TheoremPairingSignatures},
which parallel  phenomena observed for $\Hilb^d\AA^2$ in \cite{Haiman98}.
As a byproduct, we also prove Theorem \ref{TheoremSmoothnessCriterion}, which characterizes smooth monomial points on the Hilbert scheme.

A monomial ideal  $I \subseteq S$ 
admits direct sum decompositions, as module over the subrings of $S$, into smaller monomial ideals.
For instance, the 
 $\kk[z]-$ and $\kk[y,z]-$decompositions of $I$ are
$$
I = \bigoplus_{i,j} x^i y^j \big(z^{b_{i,j}}\big)  = \bigoplus_{i} x^i I_i
$$
where $(z^{b_{i,j}})\subseteq \kk[z]$ and $I_i \subseteq \kk[y,z]$ are  monomial ideals.
Clearly, such decompositions exist and are unique.
Since $I$ is an  ideal, we have $b_{i,j} \geq b_{i+1,j}, b_{i,j+1}$ and $I_i \subseteq I_{i+1}$.
If $I$ is $\mm$-primary, then $b_{i,j}=0$ for all but finitely many pairs $i,j$,
and $I_i = \kk[y,z]$ for all but finitely many $i$.
Analogous remarks hold  for the  $\kk[x]-,$ $\kk[y]-,$ $\kk[x,y]-,$ and $\kk[x,z]-$decompositions of $I$.

\begin{remark}
Let $[I]\in \Hilb^d\AA^2$ be a monomial point.
In his way to proving that $\Hilb^d\AA^2$ is smooth, Haiman \cite{Haiman98} shows that
\begin{equation}\label{EquationPnNpA2}
\dim_\kk T_{\p\n}(I) = \dim_\kk T_{\n\p}(I) =d.
\end{equation}
In fact,  a more precise statement is proved. 
Consider the $\kk[y]$-decomposition  $I =\bigoplus x^i (y^{b_i})$.
Then for each $i\in \N$ we have 
\begin{equation}\label{EquationBiA2}
\sum_{\alpha_1 = i} \dim_\kk  | T(I)|_\alpha = \sum_{\alpha_1 = -i-1} \dim_\kk  | T(I)|_\alpha = b_i.
\end{equation}
\end{remark}

Equations \eqref{EquationPnNpA2} and \eqref{EquationBiA2} cannot be extended directly to $\AA^3$,
since the Hilbert scheme is singular,
and the  dimension of $T(I)$ actually depends on $I$ and not just on $d$.
Nevertheless, 
we are going to establish versions of these equations for $\Hilb^d\AA^3$.

We begin with a homological lemma, which 
 we state  in the general case of a polynomial ring in $n$ variables, for simplicity.

\begin{lemma}\label{LemmaSerreDuality}
Let $S = \kk[x_1, \ldots, x_n]$ and $M$ be an Artinian $\Z^n$-graded $S$ module.
For each $\ell=0, \ldots, n$ there is a natural isomorphism of functors of finitely generated $\Z^n$-graded $S$ modules 
$$
\Ext^\ell_S(-,M) \cong \Ext^{n-\ell}_S(M,-\otimes \omega_S)' 
$$
where  $-'$ denotes the Matlis dual and $\omega_S $   the $\Z^n$-graded canonical module.
In particular, for every finitely generated $\Z^n$-graded module $N$ we have 
$$
\Ext^\ell_S(N,M)^\vee \cong \Ext^{n-\ell}_S(M,N)(-1,-1,\ldots,-1)
$$
as $\Z^n-$graded vector spaces, where $-^\vee$ denotes the $\kk$-dual.
\end{lemma}
\begin{proof}
To prove the first assertion,
by the universal properties of derived functors \cite[A.3.9]{Eisenbud}, 
it suffices to verify the following four properties for the functors $ \Ext^{n-\ell}_S(M,-\otimes \omega_S)'$.
\begin{enumerate}
\item  \label{ItemExtDual1}
Isomorphism for $\ell = 0$, that is, 
$
\Hom_S(-,M) \cong \Ext^{n}_S(M,-\otimes \omega_S)'.
$
\item \label{ItemExtDual2}
The vanishing $\Ext^{n-\ell}_S(M,P\otimes \omega_S)'=0$  for finitely generated projective $P$ and $\ell >0$.
\item \label{ItemExtDual3}
For each short exact sequence $0 \rightarrow N' \rightarrow N \rightarrow N'' \rightarrow 0$,
there is a long exact sequence of $ \Ext^{n-\ell}_S(M,-\otimes \omega_S)'$.
\item \label{ItemExtDual4}
Naturality  of the connecting homomorphism, that is, 
for each map of  short exact sequences of $S$-modules, the two long exact sequences of $ \Ext^{n-\ell}_S(M,-\otimes \omega_S)'$ form a  a commutative diagram.
\end{enumerate}
For \eqref{ItemExtDual1},
observe that we have natural isomorphisms
$$
\Hom_S(-,M) \cong  \Hom_S\big(-,\Ext^n_S(M,\omega_S)'\big)
\cong  \big(- \otimes \Ext^n_S(M,\omega_S)\big)'
\cong  \Ext^n_S(M, - \otimes\omega_S)'.
$$
The first one follows from the  Local Duality Theorem  \cite[3.6]{BrunsHerzog}, while 
the second one by Hom-Tensor adjointness.
To see  the third one,
let  $F_\bullet$ be a minimal free resolution of $M$, then
\begin{align*}
\Ext^n_S(M, - \otimes\omega_S) &= H^n\big( \Hom(F_\bullet, - \otimes\omega_S)\big) \\
&\cong H^n\big( \Hom(F_\bullet, \omega_S ) \otimes -\big) \\
&\cong H^n\big( \Hom(F_\bullet, \omega_S ) \big) \otimes -  && \text{by right-exactness}\\
&=  \Ext^n_S(M, \omega_S)  \otimes -.
\end{align*}

For \eqref{ItemExtDual2} it suffices to show the vanishing in the case $P=S$, which 
 follows since $M$ is Cohen-Macaulay of grade $n$, cf. \cite[3.3]{BrunsHerzog}. 
Items \eqref{ItemExtDual3} and \eqref{ItemExtDual4}
 follow from the corresponding properties of $\Ext^\bullet_S(M, -)$ combined with the exact contravariant functor $-'$.
Finally,
the second assertion of the theorem follows from the first since
$\omega_S \cong S(-1, -1, \ldots, -1)$ and  $-^\vee \cong -'$, cf.   \cite[3.6]{BrunsHerzog}.
\end{proof}

\begin{prop}\label{PropositionDualityBij}
Let $[I]\in \Hilb^d\AA^3$ be a monomial point, 
with $\kk[z]-$decomposition  $I = \bigoplus x^iy^j(z^{b_{i,j}})$. 
For every $i,j\in \N$ we have
\begin{equation}\label{EquationBij}
\sum_{\substack{\alpha_1 = i\\ \alpha_2 = j }} \dim_\kk  | T(I)|_\alpha = 
b_{i,j}+ \sum_{\substack{\alpha_1 = -i-1\\ \alpha_2 = -j-1 }} \dim_\kk  | T(I)|_\alpha.
\end{equation}
\end{prop}

\begin{proof}
Fix $i,j\in \N$ and 
consider the groups $\Ext^\ell_S(S/I,S/I)$ for $\ell = 0, \ldots, 3$.
We have 
$$
\Ext^0_S(S/I,S/I) = S/I
\quad \text{and} \quad
\Ext^1_S(S/I,S/I) =T(I),
$$
where the latter holds since $\Ext^1_S(S/I,S/I) = \Ext^0_S(I,S/I)$  by
homological ``dimension shift''.
By Lemma \ref{LemmaSerreDuality} we have $\Ext_S^\ell(S/I,S/I)^\vee \cong \Ext_S^{3-\ell}(S/I,S/I)(-1,-1,-1)$,
hence
\begin{align*}
\sum_{\substack{\alpha_1 = i\\ \alpha_2 = j }} \dim_\kk  \left| \Ext^0_S(S/I,S/I)\right|_\alpha &=  
\sum_{\substack{\alpha_1 = i\\ \alpha_2 = j }} \dim_\kk  \left| S/I\right|_\alpha = 
b_{i,j},\\
\sum_{\substack{\alpha_1 = i\\ \alpha_2 = j }} \dim_\kk  \left| \Ext^1_S(S/I,S/I)\right|_\alpha &= \sum_{\substack{\alpha_1 = i\\ \alpha_2 = j }} \dim_\kk  | T(I)|_\alpha,\\
\sum_{\substack{\alpha_1 = i\\ \alpha_2 = j }} \dim_\kk  \left| \Ext^2_S(S/I,S/I)\right|_\alpha &= \sum_{\substack{\alpha_1 = -i-1\\ \alpha_2 = -j-1 }} \dim_\kk  | T(I)|_\alpha,\\
\sum_{\substack{\alpha_1 = i\\ \alpha_2 = j }} \dim_\kk  \left| \Ext^3_S(S/I,S/I)\right|_\alpha &= 
\sum_{\substack{\alpha_1 = -i-1\\ \alpha_2 = -j-1 }} \dim_\kk  | S/I|_\alpha = 0.
\end{align*}
Equation \eqref{EquationBij} is then equivalent to
\begin{equation}\label{EquationEuler}
\sum_{\ell=0}^3 (-1)^\ell \sum_{\substack{\alpha_1 = i\\ \alpha_2 = j }} \dim_\kk  \big| \Ext^\ell_S(S/I,S/I)\big|_\alpha = 0.
\end{equation}

Let $I = (\bfx^{\beta^{(1)}}, \ldots, \bfx^{\beta^{(m)}})$
 and let $F_\bullet$ be the Taylor free resolution of $S/I$
\cite[4.3.2]{MillerSturmfels}.
The  modules in $F_\bullet$ are given by
\begin{equation*}\label{EquationTaylor}
F_\ell = \bigoplus_{\substack{\mathcal{A}\subseteq\{1, \ldots,m\}\\ \Card(\mathcal{A}) = \ell }}  S \big( - \beta^{\mathcal{A}}\big)
\qquad \text{where} \quad
\bfx^{\beta^{\mathcal{A}}} = \lcm \big\{\bfx^{\beta^{(a)}} \, : \, a \in \mathcal{A}\big\}.
\end{equation*}
Since $\Ext^\ell_S(S/I, S/I) = H^\ell\big( \Hom_S(F_\bullet, S/I)\big) = H^\ell\big( \Hom_{S/I}(F_\bullet/IF_\bullet, S/I)\big)$,
we can rephrase \eqref{EquationEuler} as
\begin{equation}\label{EquationEuler2}
\sum_{\ell=0}^m (-1)^\ell \sum_{\substack{\alpha_1 = i\\ \alpha_2 = j }} \dim_\kk  \left| \Hom_{S/I}(F_\ell/IF_\ell, S/I)\right|_\alpha = 0.
\end{equation}
Define for  each $\mathcal{A}\subseteq \{1, \ldots, m\}$ the quantity
$$
t_\mathcal{A} = \sum_{\substack{\alpha_1 = i\\ \alpha_2 = j }} \dim_\kk  \big| \Hom_{S/I}\big( {S}/{I}\big(-\beta^{\mathcal{A}}\big), S/I\big)\big|_\alpha.
$$
then \eqref{EquationEuler2} is equivalent to 
\begin{equation}\label{EquationEuler3}
\sum_{\mathcal{A} \subseteq \{1, \ldots, m\}} (-1)^{\Card(\mathcal{A})}t_\mathcal{A} = 0.
\end{equation}
Note that for each $\alpha$ and $\mathcal{A}$ we have
$$
\big| \Hom_{S/I}\big( {S}/{I}\big(-\beta^{\mathcal{A}}\big), S/I\big)\big|_\alpha = 
 \big| \Hom_{S/I}\big( {S}/{I}, S/I\big)\big|_{\alpha + \beta^{\mathcal{A}}}
=  \big| {S}/{I}\big|_{\alpha + \beta^{\mathcal{A}}}
$$
so that 
$$
\dim_\kk\big| \Hom_{S/I}\big( {S}/{I}\big(-\beta^{\mathcal{A}}\big), S/I\big)\big|_\alpha =
 \begin{cases} 
1 &\mbox{if } \alpha + \beta^{\mathcal{A}} \in \N^3 \setminus \tilde{I},\\ 
0  &\mbox{otherwise.} 
\end{cases}
$$
Adding over all $\alpha_3 \in \Z$ 
we get
$t_\mathcal{A} = \Card \big\{ \alpha_3 \in \Z \, : \, (i,j,\alpha_3) + \beta^{\mathcal{A}} \in \N^3 \setminus \tilde{I}\big\}$,
that is,
in terms of the  $\kk[z]-$decomposition of $I$,
\begin{equation}\label{EquationTA}
t_\mathcal{A} = b_{i+\beta^{\mathcal{A}}_1, j+\beta^{\mathcal{A}}_2}.
\end{equation}
Assuming without loss of generality that $\bfx^{\beta^{(m)}} = z^{b_{0,0}}$, 
the formula \eqref{EquationTA} immediately implies that $t_\mathcal{A} = t_{\mathcal{A}\cup\{m\}}$ for every $\mathcal{A}$, 
which in turn yields \eqref{EquationEuler3} and concludes the proof.
\end{proof}

The following consequence of Proposition \ref{PropositionDualityBij} is the main result of the section.

\begin{thm}\label{TheoremPairingSignatures}
Let $[I]\in \Hilb^d\AA^3$ be a monomial point.
We have
\begin{eqnarray*}
\dim_\kk  T_{\ppn}(I) = \dim_\kk  T_{\nnp}(I) + d, \\
\dim_\kk  T_{\pnp}(I) = \dim_\kk  T_{\npn}(I) + d,  \\
\dim_\kk  T_{\npp}(I) = \dim_\kk  T_{\pnn}(I) + d.
\end{eqnarray*}
\end{thm}

\begin{proof}
The first equation follows from Proposition \ref{PropositionDualityBij} 
by adding over all $i,j \in \N$, and using Proposition \ref{PropositionDecompositionDistinguishedSubspaces}.
The other two follow from the first  by permutation.
\end{proof}

Theorem \ref{TheoremPairingSignatures} provides the correct generalization of  \eqref{EquationPnNpA2} to $\AA^3$, since 
it implies
$$
\dim_\kk  T_{\p\n\ast}(I) = \dim_\kk  T_{\n\p\ast}(I), \quad
\dim_\kk  T_{\p\ast\n}(I) = \dim_\kk  T_{\n\ast\p}(I), \quad
\dim_\kk  T_{\ast\p\n}(I) = \dim_\kk  T_{\ast\n\p}(I),
$$
where   e.g. $ T_{\p\n\ast}(I) = T_{\pnp}(I) \oplus T_{\pnn}(I)$.
To the best of our knowledge, Proposition \ref{PropositionDualityBij} and Theorem \ref{TheoremPairingSignatures} do not extend to higher dimensions.

Theorem \ref{TheoremPairingSignatures} is  also a vast generalization of the following parity result of Maulik, Nekrasov, Okounkov, and Pandharipande, which follows from \cite[Theorem 2]{MNOP},
see also \cite[Lemma 4.1 (c)]{BehrendFantechi}.

\begin{cor}\label{CorollaryParity}
For each monomial point  $[I]\in \Hilb^d\AA^3$ we have
$
\dim_\kk T(I) \equiv d \bmod 2 
$.
\end{cor}

Whether
$
\dim_\kk T(I) \equiv d \bmod 2 
$
for every 
$[I]\in \Hilb^d\AA^3$ is an open and interesting question;
see  \cite[Remark 22]{BryanKool}
 for related matters.
 A stronger open question is whether  for any  $[I]\in \Hilb^d\AA^3$ there exists a monomial $[M]\in \Hilb^d\AA^3$  such that 
$\dim_\kk T(I) = \dim_\kk T(M)$.

Another interesting special case of  Theorem \ref{TheoremPairingSignatures} occurs when each of the three equations is a small as possible: we obtain the following smoothness criterion for  monomial points in $\Hilb^d\AA^3$.

\begin{thm} \label{TheoremSmoothnessCriterion}
A monomial point $[I] \in \Hilb^d \AA^3 $ is smooth  if and only if 
$$
T_\mathfrak{s}(I)=0
\qquad
\text{for }\mathfrak{s} \in \{ \p\n\n, \n\p\n, \n\n\p\}.
$$
\end{thm}

\begin{proof}
It is known that a monomial point $[I]$  lies in the closure of the component of $\Hilb^d\AA^3$ parametrizing subschemes of  $d$ distinct points, see e.g.
 \cite[4.15]{CEVV}.
We deduce that  $[I]$ is a smooth point   if and only if $\dim_\kk T(I) = 3d$, 
and the statement follows by Theorem \ref{TheoremPairingSignatures}.
\end{proof}

The criterion can be particularly effective in proving that a point $[I]$ is singular:
 it suffices to exhibit a single tangent vector with forbidden signature.
In many cases, 
the existence of such tangent vector follows just by looking at the minimal generators of $I$.
We give two examples.

\begin{cor}\label{CorollarySingularByGenerators}
Let $[I] \in \Hilb^d \AA^3 $ be a monomial point.
Suppose the  minimal generating set of $I$ contains three monomials 
$
x^{\alpha_1}y^{\alpha_2}, 
x^{\beta_1}z^{\beta_3},
y^{\gamma_2}z^{\gamma_3}
 $
with 
$\alpha_1, \alpha_2, \beta_1, \beta_3, \gamma_2, \gamma_3 >0$
satisfying one of the following:
\begin{itemize}
\item $\alpha_1 \geq \beta_1$ and  $\alpha_2 \geq \gamma_2$;
\item $\beta_1 \geq \alpha_1$ and  $\beta_3 \geq \gamma_3$;
\item $\gamma_2 \geq \alpha_2$ and  $\gamma_3 \geq \beta_3$.
\end{itemize}
Then $[I]$ is a singular point.
\end{cor}
\begin{proof}
Since $\dim_\kk(S/I) < \infty$, 
there are also  minimal generators $x^{s_1}, y^{s_2}, z^{s_3}$,
and by minimality we get $s_1 > \alpha_1, \beta_1$, $s_2 > \alpha_2, \gamma_2$, $s_3> \beta_3, \gamma_3$. 
It follows that there are monomials
$$
x^{\delta_1}y^{\delta_2}z^{s_3-1},\quad
x^{\epsilon_1}y^{s_2-1}z^{\epsilon_3},\quad
x^{s_1-1}y^{\zeta_2}z^{\zeta_3}\quad
\in \soc\left( \frac{S}{I}\right)
$$
for some $\delta_1 \leq \beta_1-1, \delta_2 \leq \gamma_2-1, \epsilon_1 \leq \alpha_1-1, \epsilon_3 \leq \gamma_3-1, \zeta_2 \leq \alpha_2-1, \zeta_3 \leq \beta_3-1$.
By Remark \ref{RemarkBasisSocle} and by \eqref{EquationSocleTangentSpace} there are three maps
$\varphi_1, \varphi_2, \varphi_3 \in \soc(T(I)) \subseteq T(I)$ such that 
$$
\varphi_1(x^{\alpha_1}y^{\alpha_2}) = x^{\delta_1}y^{\delta_2}z^{s_3-1}, \quad
\varphi_2(x^{\beta_1}z^{\beta_3}) = x^{\epsilon_1}y^{s_2-1}z^{\epsilon_3}, \quad
\varphi_3(y^{\gamma_2}z^{\gamma_3}) = x^{s_1-1}y^{\zeta_2}z^{\zeta_3}.
$$
Using the hypothesis we derive
$\varphi_1 \in T_{\n\n\p}(I),$ or $\varphi_2 \in T_{\n\p\n}(I),$ or $\varphi_3 \in T_{\p\n\n}(I)$.
\end{proof}

\begin{cor}\label{CorollarySingularStronglyStable}
Let $[I] \in \Hilb^d \AA^3 $ be a strongly stable point.
Then $[I]$ is smooth  if and only if $x \in I$.
\end{cor}
\begin{proof}
Assume $x\notin I$ and let $z^{s_3} \in I$ be a minimal generator.
By strong stability,  $xy^{a}$ is a minimal generator for some $a>0$, and moreover $xz^{s_3-1}, yz^{s_3-1} \in I$,
thus $z^{s_3-1} \in \soc(S/I)$.
By Remark \ref{RemarkBasisSocle} and by \eqref{EquationSocleTangentSpace} there is a map
$\varphi\in \soc(T(I))\subseteq T(I)$ such that $\varphi(xy^{a}) = z^{s_3-1}$, so $\varphi \in T_{\n\n\p}(I)\ne 0$.

Now assume $x\in I$.
Then $\gamma_1 = 0$ for all $\bfx^\gamma \in S/I$,
and  $\beta_1=0$ for all generators $ \bfx^\beta \ne x$ of $I$.
Let $\varphi  \in |T(I)|_\alpha$ for some $\alpha$.
If $\varphi(x) \ne 0$ then $\alpha_2, \alpha_3 \geq 0$, so $\varphi \in T_\npp(I)$.
Suppose $\varphi(\bfx^\beta) \ne 0$ for some generator $x \ne \bfx^\beta \in I$, then  $\alpha_1 = 0 $ since $\beta_1=0$.
Assume by contradiction that $\alpha_2, \alpha_3 <0$.
Considering the ``boundary'' 
$B = \tilde{I}\,\setminus \big( \tilde{I} + (0, 1, 1) \big)$
and arguing as in Proposition \ref{PropositionDecompositionDistinguishedSubspaces},
we see that $(\tilde{I}+\alpha)\setminus \tilde{I}$ is connected and unbounded.
This contradicts Proposition \ref{PropositionBasis},
thus $\alpha_2 \geq 0$ or $\alpha_3 \geq 0$, and  $\varphi \in T_{\ppn}(I)\oplus T_{\pnp}(I)$.
We conclude that $T_{\p\n\n}(I) =T_{\n\p\n}(I) = T_{\n\n\p}(I)=0$.
\end{proof}

\section{Extremality of subspaces of the tangent space}\label{SectionExtremality}

In this section we prove
Theorem \ref{TheoremExtremalSubspaces},
confirming the extremal behavior predicted by Conjecture \ref{BIconj1}
for certain components $T_\mathfrak{s}(I)$ of the tangent space.

\begin{prop} \label{PropositionBoundPPN} 
Let  $[I] \in \Hilb^d \AA^3$ be a monomial point with 
$\kk[z]-$decomposition $I=\bigoplus x^iy^j\left(z^{b_{i,j}}\right)$. 
For each $\alpha_1,\alpha_2 \geq 0$ we have the inequality
\begin{equation*}
 \sum_{\alpha_3<0} \dim_\kk \big|T(I)\big|_{(\alpha_1,\alpha_2,\alpha_3)} 
 \leq \sum_{\substack{i \geq \alpha_1\\ j \geq \alpha_2}} \left(b_{i,j} - \max\{b_{i+1,j},b_{i,j+1}\}\right).
\end{equation*}
\end{prop}
\begin{proof}
Fix non-negative integers $\alpha_1, \alpha_2$,
and define the sets
\begin{align*}
\mathcal{C} &= \bigcup_{\alpha_3 <0} \Big\{\text{bounded connected components of } \left(\tilde{I}+ (\alpha_1,\alpha_2,\alpha_3)\right)  \setminus \, \tilde{I} \Big\},\\
\mathcal{S} & = \Big\{(i,j,k)\notin \tilde{I}\,:\, i \geq \alpha_1, j \geq \alpha_2 \text{ and }(i+1,j,k),(i,j+1,k)\in \tilde{I}\Big\}.
\end{align*}
We define a map $\Psi : \mathcal{C} \rightarrow \mathcal{S}$ by choosing, 
for each $U \in \mathcal{C}$,  a vector $\Psi(U) = (\psi^U_1, \psi^U_2, \psi^U_3)\in U $ such that $\psi^U_3$ is the least possible among  vectors in $ U$, and  
$(\psi^U_1+1, \psi^U_2, \psi^U_3), (\psi^U_1, \psi^U_2+1, \psi^U_3)\notin U$. 
The choice is possible as $\Card(U) < \infty$.
Since $U$ is a bounded  connected component of $(\tilde{I}+ (\alpha_1,\alpha_2,\alpha_3)) \setminus \tilde{I}$ for some $\alpha_3$,
the vector $\Psi(U)$ is indeed in $\mathcal{S}$.

We claim that the map $\Psi$ is injective.
Let  $U\ne U'$ be bounded  components of $\left(\tilde{I}+ (\alpha_1,\alpha_2,\alpha_3)\right)\setminus \, \tilde{I}$  and $\left(\tilde{I}+ (\alpha_1,\alpha_2,\alpha'_3)\right)\setminus \, \tilde{I}$, respectively, for some $\alpha_3, \alpha'_3 <0$.
If $\alpha_3 = \alpha'_3$ then $U\cap U' = \emptyset$  by definition of connected component, hence $\Psi(U) \ne \Psi({U'})$.
Suppose now $\alpha_3 < \alpha'_3$,
thus
$\left(\tilde{I}+ (\alpha_1,\alpha_2,\alpha'_3)\right)\setminus \, \tilde{I}\subsetneq \left(\tilde{I}+ (\alpha_1,\alpha_2,\alpha_3)\right)\setminus \, \tilde{I}$.
If $U\cap U' \ne \emptyset$ then necessarily $U' \subsetneq U$, and this implies 
$\Psi(U') + (0,0,  \alpha_3 - \alpha_3') \in U$.
We conclude that $\psi^U_3 \leq \psi^{U'}_3 + \alpha_3 - \alpha'_3< \psi^{U'}_3$, in particular $\Psi(U) \ne \Psi({U'})$ as claimed.

Note that, for each pair $i,j$, we have  $\Card\big\{(i,j,k)\notin \tilde{I}\,:\, (i+1,j,k),(i,j+1,k)\in \tilde{I}\big\}= b_{i,j} - \max\{b_{i+1,j},b_{i,j+1}\}$.
We deduce that 
$$
\Card(\mathcal{C}) \leq \Card(\mathcal{S})=\sum_{\substack{i \geq \alpha_1\\ j \geq \alpha_2}} \left(b_{i,j} - \max\{b_{i+1,j},b_{i,j+1}\}\right)
$$
concluding the proof by Proposition \ref{PropositionBasis}.
 \end{proof}

By combining the inequalities for all $\alpha_1, \alpha_2\geq 0$ Proposition \ref{PropositionBoundPPN} provides  upper bounds for $T_{\ppn}(I)$ and, 
up to permutation, for $T_{\pnp}(I)$ and $T_{\npp}(I)$.
Using the symmetries of Section \ref{SectionSymmetry},
we also obtain estimates for the other three signatures.
We are going to apply these  bounds to Borel-fixed points.
Before we can present the main result, we need  some lemmas 
about strongly stable ideals and powers of $\mm$.

\begin{lemma} \label{LemmaSimplifyMax} 
Let    $[I] \in \Hilb^d \AA^3$ be  a strongly stable point with $\kk[z]-$ and $\kk[y]-$decompositions
 $$
 I=\bigoplus x^iy^j\big(z^{b^z_{i,j}}\big)=\bigoplus x^iz^j\big(y^{b^y_{i,j}}\big).
 $$
Then $\max\{b^{z}_{i+1,j},b^{z}_{i,j+1}\} = b^{z}_{i,j+1}$ and $\max\{b^{y}_{i+1,j},b^{y}_{i,j+1}\} = b^{y}_{i,j+1}$ for all $i,j$.
\end{lemma}
\begin{proof} 
Since $I$ is strongly stable,
$x^iy^{j+1}z^{b^{z}_{i,j+1}} \in I$ implies $x^{i+1}y^{j}z^{b^{z}_{i,j+1}} \in I$,
 thus, by definition $b^{z}_{i+1,j} \leq b^{z}_{i,j+1}$ i.e. $\max\{b^{z}_{i+1,j},b^{z}_{i,j+1}\} = b^{z}_{i,j+1}$. 
 The other equation is proved similarly. 
\end{proof}

\begin{lemma} \label{LemmaBorelComponents}
Let    $[I] \in \Hilb^d \AA^3$ be  a strongly stable point with 
$\kk[y,z]-$decomposition $I = \bigoplus x^i I_i$.
Then for every $i\geq 0$ the ideal $I_i$ is strongly stable, and we have $I_i : y \subseteq I_{i+1}$.
\end{lemma}
\begin{proof}
Both properties follow easily by strong stability.
\end{proof}

\begin{lemma} \label{LemmaExtremalSequence}
Let    $[I] \in \Hilb^d \AA^3$ be  a strongly stable point with 
$\kk[y,z]-$decomposition $I = \bigoplus x^i I_i$.
If $d\leq \dim_\kk(S/\mathfrak{m}^r)$ then for all $0\leq j \leq r$ we have 
\begin{equation}\label{InequalityPartialSequence}
\sum_{i=j}^{r-1} \dim_\kk \frac{\kk[y,z]}{I_i} \leq \sum_{i=j}^{r-1} \dim_\kk \frac{\kk[y,z]}{(y,z)^{r-i}}.
\end{equation}
Moreover, if equality holds for all $0\leq j \leq r-1$ then $I= {\mm}^r$.
\end{lemma}
\begin{proof} 
Observe that $\mm^r$ has $\kk[y,z]-$decomposition $\mm^r = \bigoplus x^i (y,z)^{r-i}$,
with the convention that  $(y,z)^h = \kk[y,z]$ if $h<0$.

Suppose first    $\dim_\kk \big({\kk[y,z]}/{I_0}\big) \geq  \dim_\kk \big({\kk[y,z]}/{(y,z)^{r}}\big)$.
We  prove the inequalities \eqref{InequalityPartialSequence} by induction on $\ell = \min\{ h \, : \, x^h \in I\}$.
The  case  $\ell = 1$  is clear, so we  assume $\ell > 1$. 
Define $I' = \bigoplus  x^{i-1}I_i \subseteq S$, then $x^{\ell-1}\in I$ and 
\begin{equation*}
\dim_\kk(S/I') = \dim_\kk(S/I) -\dim_\kk\big(\kk[y,z]/I_0\big) \leq \dim_\kk(S/\mm^r) - \dim_\kk\big(\kk[y,z]/(y,z)^r\big) = \dim_\kk(S/{\mm}^{r-1}).
\end{equation*}
Applying the inductive step to $I'$ and $\mm^{r-1}$ we deduce 
\begin{equation*}
\sum_{i=j}^{r-1} \dim_\kk\frac{\kk[y,z]}{I_i} = \sum_{i=j-1}^{r-2} \dim_\kk\frac{\kk[y,z]}{I'_i}  \leq \sum_{i=j-1}^{r-2} \dim_\kk\frac{\kk[y,z]}{(y,z)^{r-1-i}} = \sum_{i=j}^{r-1} \dim_\kk\frac{\kk[y,z]}{(y,z)^{r-i}}.
\end{equation*}
verifying \eqref{InequalityPartialSequence} for all $1 \leq j \leq r-1$, while the case $j=0$ holds by assumption. 

Suppose  now that  $\dim_\kk \big({\kk[y,z]}/{I_0}\big) \leq  \dim_\kk \big({\kk[y,z]}/{(y,z)^{r}}\big)$.
We claim that $\dim_\kk \big({\kk[y,z]}/{I_i}\big) \leq  \dim_\kk \big({\kk[y,z]}/{(y,z)^{r-i}}\big)$ for all $i$,
implying  the  inequalities \eqref{InequalityPartialSequence}.
By Lemma \ref{LemmaBorelComponents} it suffices to verify  the following statement:
if  $J\subseteq \kk[y,z]$ is strongly stable 
and $\dim_\kk \big({\kk[y,z]}/{J}\big) \leq  \dim_\kk \big({\kk[y,z]}/{(y,z)^{h}}\big)$ for some $h$,
then $\dim_\kk \big({\kk[y,z]}/{(J:y)}\big) \leq  \dim_\kk \big({\kk[y,z]}/{(y,z)^{h-1}}\big)$.
Write
$
J = (y^{a},y^{a-1}z^{c_1}, y^{a-2}z^{c_2}, \ldots, yz^{c_{a-1}}, z^{c_a}),
$
so 
$
J :y = (y^{a-1},y^{a-2}z^{c_1}, y^{a-3}z^{c_2}, \ldots, z^{c_{a-1}}).
$
If $c_a \leq h$ then 
 $(y,z)^h \subseteq J$ by strong stability,
thus $(y,z)^{h-1} = (y,z)^{h}:y  \subseteq J:y$ and the claim follows.
If $c_a > h$ 
then the claim follows as
$$
\dim_\kk \frac{\kk[y,z]}{J} - \dim_\kk \frac{\kk[y,z]}{J:y}= 
\sum_{i=1}^a c_i - \sum_{i=1}^{a-1} c_i =
c_a > h =  \dim_\kk \frac{\kk[y,z]}{(y,z)^h} - \dim_\kk \frac{\kk[y,z]}{(y,z)^{h-1}}.
$$

Finally, assume equality holds  in \eqref{InequalityPartialSequence} for all $j$,
 then  $\dim_\kk \big(\kk[y,z]/{I_i}\big)=\dim_\kk \big({\kk[y,z]}/{(y,z)^{r-i}}\big)$ for all $i$.
We show by decreasing  induction on $i$ that $I_i = (y,z)^{r-i}$.
If $x^r \notin I$ then $I$ contains no monomial of degree $r$, by strong stability, yielding the contradiction $I \subseteq \mm^{r+1}$.
Thus $I_i = \kk[y,z]$ for all $i \geq r$.
Now suppose $I_i = (y,z)^{r-i}$ for some $0<i \leq r$.
Using the argument of the previous paragraph with $J= I_{i-1}$ and $h= r-i+1$, we 
must have $c_a \leq  h$,
otherwise $\dim_\kk \big(\kk[y,z]/{I_i}\big) \leq \dim_\kk \big(\kk[y,z]/({I_{i-1}}:y)\big) < \dim_\kk \big({\kk[y,z]}/{(y,z)^{r-i}}\big)$.
But if $c_a \leq  h$ then 
 $(y,z)^{r-i+1}=(y,z)^h \subseteq J= I_{i-1}$,
 and equality must hold by dimension reasons.
\end{proof}

\begin{lemma}\label{LemmaFatPointPPN}
Let $r \in \N$.
We have 
\begin{eqnarray*}
\dim_\kk T_\ppn(\mm^r)= \dim_\kk T_\pnp(\mm^r) = \dim_\kk T_\npp(\mm^r) &=& {r+3 \choose 4},\\
\dim_\kk T_\pnn(\mm^r) = \dim_\kk T_\npn(\mm^r) = \dim_\kk T_\nnp(\mm^r) &= &{r+2\choose 4}.
\end{eqnarray*}
In particular, $\dim_\kk T(\mm^r) = {r+2 \choose 2}{r+1 \choose 2}.$
\end{lemma}

\begin{proof}
Using Lemma \ref{LemmaHilbertFunctionFatPoint} and the ``hockey-stick identity'' of binomial coefficients  one gets
\begin{align*}
\dim_\kk T_\ppn(\mm^r) & = \sum_{\substack{\alpha_1, \alpha_2 \geq 0, \alpha_3 \geq -r  \\ \alpha_1+\alpha_2 +\alpha_3 = -1}} {r+2 +\alpha_3 \choose 2} =
\sum_{\alpha_1 = 0}^{r-1} \sum_{\alpha_2 = 0}^{r-1-\alpha_1} {r+1-\alpha_1-\alpha_2 \choose 2} 
=
\sum_{\alpha_1 = 0}^{r-1} \sum_{h = 2}^{r+1-\alpha_1} {h \choose 2} \\
& =
\sum_{\alpha_1 = 0}^{r-1} {r+2-\alpha_1  \choose 3} 
=
\sum_{k = 3}^{r+2} {k \choose 3} 
=
 {r+3 \choose 4}. 
\end{align*}
The same holds for $T_\pnp(\mm^r),  T_\npp(\mm^r)$ by symmetry.
The other formula is proved likewise.
The last formula follows from Proposition \ref{PropositionDecompositionDistinguishedSubspaces}.
\end{proof}

We are now ready to state the main theorem of this section:
\begin{thm} \label{TheoremExtremalSubspaces} Let $\mathrm{char}\, \kk = 0$ and $[I] \in \Hilb^d \mathbf{A}^3$ be  Borel-fixed,
with $d = { r+2 \choose 3}$.  Then we have
\begin{eqnarray*}
\dim_\kk  T_{\ppn}(I) \leq \dim_\kk T_{\ppn}(\mm^r), & \, & \dim_\kk  T_{\pnp}(I) \leq \dim_\kk T_{\pnp}(\mm^r), \\
\dim_\kk  T_{\nnp}(I) \leq \dim_\kk T_{\nnp}(\mm^r), & \, & \dim_\kk  T_{\npn}(I) \leq \dim_\kk T_{\npn}(\mm^r). 
\end{eqnarray*}
Moreover, in each case equality occurs   if and only if $I = \mathfrak{m}^r$.
\end{thm}

\begin{proof} 
By Theorem \ref{TheoremPairingSignatures}
 it suffices to prove the first two inequalities.
We consider the $\kk[z]-$, $\kk[y]-$ and  $\kk[y,z]-$decompositions of $I$
$$
I = \bigoplus x^iy^j\Big(z^{b^{z}_{i,j}}\Big) = \bigoplus x^iz^j\Big(y^{b^{y}_{i,j}}\Big)
= \bigoplus x^iI_i.
$$
Note that $\sum_{j\geq 0} b^{z}_{i,j} = \dim_\kk(\kk[y,z]/I_i)$ for each $i$.
Recall that $I_i = \kk[y,z]$ for all $i \geq r$, as observed in the proof of Lemma \ref{LemmaExtremalSequence}. 
We apply  Proposition \ref{PropositionBoundPPN} and Lemmas \ref{LemmaSimplifyMax},  \ref{LemmaExtremalSequence},
\ref{LemmaFatPointPPN} to obtain
\begin{align*}
\dim_\kk  T_{\ppn}(I) 
&= 
\sum_{\substack{ \alpha_1 , \alpha_2 \geq 0 \\ \alpha_3 <0}} \dim_\kk  |T(I)|_{(\alpha_1, \alpha_2, \alpha_3)} 
\leq  \sum_{\substack{\alpha_1 , \alpha_2 \geq 0}} \sum_{\substack{i \geq \alpha_1\\  j \geq \alpha_2}} \big(b_{i,j}^z - \max\{b_{i+1,j}^z,b_{i,j+1}^z\}\big) 
\\
&= 
\sum_{\substack{\alpha_1 \geq 0 \\ \alpha_2 \geq 0}} \sum_{\substack{i \geq \alpha_1\\  j \geq \alpha_2}} \big(b_{i,j}^z - b_{i,j+1}^z\big) 
= 
\sum_{i,j} (i+1)(j+1)\big(b_{i,j}^z - b_{i,j+1}^z\big) 
= 
\sum_{i,j} (i+1)b_{i,j}^z \\
&= 
\sum_{i=0}^{r-1} (i+1)\dim_\kk\frac{\kk[y,z]}{I_i}
=  
\sum_{i=0}^{r-1} \sum_{j=i}^{r-1} \dim_\kk\frac{\kk[y,z]}{I_j}
 \leq 
 \sum_{i=0}^{r-1} \sum_{j=i}^{r-1} \dim_\kk \frac{\kk[y,z]}{(y,z)^{r-j}}
 \\
& =
 \sum_{i=0}^{r-1} \sum_{j=i}^{r-1} {r-j+1\choose 2}
 =
 \sum_{i=0}^{r-1} \sum_{h=2}^{r-i+1} {h\choose 2}
 =
 \sum_{i=0}^{r-1}  {r-i+2\choose 3}
 =
 \sum_{k=3}^{r+2}  {k\choose 3}
 = {r+3 \choose 4}
 \\
& =
 \dim_\kk  T_{\ppn}(\mm^r).
\end{align*}

The  inequality $\dim_\kk  T_{\pnp}(I) \leq \dim_\kk T_{\pnp}(\mm^r)$ 
is proved in the same way, using the second part of Lemma \ref{LemmaSimplifyMax} and the fact that  for each $i$ we have $\sum_{j\geq 0} b^{y}_{i,j} = \dim_\kk(\kk[y,z]/I_i)$.

Finally, 
we verify the last assertion of the theorem.
Observe that, if any of the four inequalities is an equality,
then all the inequalities in the application of Lemma \ref{LemmaExtremalSequence} are equalities, 
so for every $0 \leq i \leq r-1$ we have
 $\sum_{j=i}^{r-1} \dim_\kk (\kk[y,z]/I_j) =  \sum_{j=i}^{r-1} \dim_\kk \big(\kk[y,z]/(y,z)^{r-j}\big)$,
  and this in turn forces $I = \mm^r$ by the second part of Lemma \ref{LemmaExtremalSequence}.
\end{proof}

\begin{remark}
By  Lemma \ref{LemmaGinReduction}, Remark \ref{RemarkCharacteristicZero}, and Proposition \ref{PropositionDecompositionDistinguishedSubspaces},
Theorem \ref{TheoremExtremalSubspaces} 
verifies two thirds of Conjecture \ref{BIconj1}
 for $\Hilb^d\AA^3$.
In fact, 
we conjecture that the remaining two inequalities 
$$
\dim_\kk  T_{\npp}(I) \leq \dim_\kk T_{\npp}(\mm^r)
\qquad \text{and}
\qquad 
\dim_\kk  T_{\pnn}(I) \leq \dim_\kk T_{\pnn}(\mm^r)
$$
are also true.
However, 
the bounds obtained for these subspaces through Proposition \ref{PropositionBoundPPN} are not sharp enough to prove them, as the next example shows.
\end{remark}

\begin{example}
Let $I = (x) +(y,z)^s$ where $s\in \N$.
We consider its $\kk[x]-$decomposition
$
I = \bigoplus y^iz^j\big(x^{b_{i,j}}\big).
$
Observe that $b_{i,j}= 1 $ if $i+j < s$, whereas $b_{i,j}= 0 $ if $i+j \geq s$. 
Proceeding as in the proof of Theorem \ref{TheoremExtremalSubspaces}, we 
use Proposition \ref{PropositionBoundPPN} to estimate $\dim_\kk  T_{\npp}(I)$ and obtain
\begin{align*}
\dim_\kk  T_{\npp}(I) 
&=
\sum_{\alpha_2 , \alpha_3 \geq 0} \sum_{\alpha_1 <0 } \dim_\kk  |T(I)|_{(\alpha_1, \alpha_2, \alpha_3)} 
\leq  \sum_{\substack{\alpha_2 , \alpha_3 \geq 0}} \sum_{\substack{i \geq \alpha_2\\  j \geq \alpha_3}} \big(b_{i,j} - \max\{b_{i+1,j},b_{i,j+1}\}\big) 
\\
&= 
\sum_{\substack{\alpha_2 , \alpha_3 \geq 0}} \sum_{\substack{\,\,i \geq \alpha_2 ,  j \geq \alpha_3\\ i+j =s-1}} 1
= 
\sum_{\substack{\alpha_2 , \alpha_3 \geq 0 \\ \alpha_2+\alpha_3 < s}} (s-\alpha_2-\alpha_3)
= {s+1 \choose 2}s - \sum_{\substack{\alpha_2 , \alpha_3 \geq 0 \\ \alpha_2+\alpha_3 < s}} (\alpha_2+\alpha_3)
\\
&= {s+1 \choose 2}s - \sum_{i=1}^{s-1} i(i+1) = {s+1 \choose 2}s -\frac{(s-1)s(s+2)}{3} = {s+2 \choose 3}.
\end{align*}
Choose  $s= 15$ and $r = 8$, 
so   $\dim_\kk (S/I) = {15 + 1 \choose 2} = 120 = {8 +2 \choose 3} = \dim_\kk (S/\mm^r)$.
The inequality above yields $\dim_\kk  T_{\npp}(I) \leq {15+2 \choose 3} = 680$;
however,  $\dim_\kk  T_{\npp}(\mm^r) = {8+3 \choose 4} = 330$ by Lemma \ref{LemmaFatPointPPN}.
\end{example}

\section{Global estimates}\label{SectionGlobalEstimates}

We now  take a more direct approach to estimating the dimension of  tangent space to a point in $\Hilb^d \AA^3$.
This section is devoted to the proof of Theorem \ref{TheoremGlobalBound}. 

Let $R$ be a regular local ring of dimension 2, and denote by $\ell(\cdot)$ the length of an $R$-module.
A key step in the proof of the smoothness of  $\Hilb^d \AA^2$ \cite{Fogarty} is 
to show that $\ell(T(I)) = 2\ell(R/I)$ for all artinian ideals $I\subseteq R$. 
The next proposition generalizes this fact.

\begin{prop} \label{Propostion2RLR} 
Let $R$ be a 2-dimensional regular local ring, and 
let $I, J\subseteq R$ be ideals satisfying $\ell(R/I), \ell(R/J) < \infty$. Then
\begin{equation*}
\ell\big(\Hom_{R}(I,R/J)\big)=\ell\big(R/J\big)+\ell\big((I:J)/I\big). 
\end{equation*}
\end{prop}
	
\begin{proof} 
Let $0 \rightarrow R^{a_2} \rightarrow R^{a_1} \rightarrow R^{a_0} \rightarrow R/I  \rightarrow 0$ be a free resolution,
then the alternating sum of ranks vanishes: $a_0-a_1+a_2 = 0$. 
Setting $\chi(R/I,R/J) = \sum_{i=0}^{2}(-1)^i\ell\big(\Ext^i(R/I,R/J)\big)$ we have 
\begin{equation}\label{EquationChi}
\chi(R/I,R/J) = \sum_{i=0}^{2} (-1)^i\chi(R^{a_i},R/J) = \sum_{i=0}^{2} (-1)^i \ell(R/J)\cdot a_i = \ell(R/J)\sum_{i=0}^{2} (-1)^i a_i = 0.
\end{equation}
Let $\omega_{R/I}$ be the canonical module of $R/I$.
Since $R/I$ is a Cohen-Macaulay $R$-module of codimension 2, 
dualizing its free resolution and 
using right-exactness of $-\otimes R/J$ yields
\begin{equation}\label{EquationExt2}
\Ext^2(R/I,R/J)\cong \Ext^2(R/I,R)\otimes R/J = \omega_{R/I} \otimes R/J.
\end{equation}
Combining  equations \eqref{EquationChi} and \eqref{EquationExt2} 
with the  exact sequence 
$$0 \rightarrow \Hom(R/I,R/J) \rightarrow R/J \rightarrow  \Hom(I,R/J) \rightarrow \Ext^1(R/I,R/J) \rightarrow 0$$
we get
\begin{eqnarray*}
\ell\big(\Hom(I,R/J)\big) & = & \ell\big(R/J\big) - \ell\big(\Hom(R/I,R/J)\big) + \ell\big(\Ext^1(R/I,R/J)\big) \\
			& = & \ell\big(R/J\big) + \ell\big(\Ext^2(R/I,R/J)\big) \\
			& = & \ell\big(R/J\big)  + \ell\big(\omega_{R/I} \otimes_R R/J\big).
\end{eqnarray*}
It remains to show that $\ell\big(\omega_{R/I}/J\omega_{R/I}\big)= \ell\big( (I:J)/I\big)$. 
We have $(I:J)/I = \big(I:(I+J)\big)/I$ and  $\omega_{R/I}/J\omega_{R/I} =\omega_{R/I}/(I+J)\omega_{R/I}$ (since $I$ annihilates $\omega_{R/I}$), 
so we may assume that  $I \subseteq J$. 
In this case $R/J$ is a finite  $R/I$-module and $\omega_{R/J} \cong \Hom(R/J, \omega_{R/I})$.
Since $\Hom(-, \omega_{R/I})$ induces a duality in the category of finite  $R/I$-modules (cf. \cite[21.1]{Eisenbud}) we obtain
\begin{eqnarray*}
\Hom(R/J,R/I) & \cong & \Hom\big(\Hom(R/I,\omega_{R/I}),\Hom(R/J,\omega_{R/I})\big) \\
			& \cong & \Hom(\omega_{R/I},\omega_{R/J}) \\
			& = & \Hom(\omega_{R/I}/J\omega_{R/I},\omega_{R/J})
\end{eqnarray*}
and this implies
$\ell\big(\Hom(R/J,R/I)\big)= \ell\big(\omega_{R/I}/J\omega_{R/I}\big)$, again by duality.
The  proof is completed, as $(I:J)/I = \Hom(R/J,R/I)$.
\end{proof}

Now we present the  main result of this section, 
which 
establishes an approximation of Conjecture \ref{BIconj1} for the Hilbert scheme of points in $\AA^3$.

\begin{thm} \label{TheoremGlobalBound} 
Let $d,r \in \N$ be such that
$ d \leq \binom{r+2}{3}$.
For all $[I] \in \Hilb^d \AA^3$ we have
\begin{equation*}
 \dim_\kk  T(I) \leq \frac{4}{3}\dim_\kk  T(\mm^r).
\end{equation*}
\end{thm}

\begin{proof}
By Remark \ref{RemarkCharacteristicZero} and Lemma \ref{LemmaGinReduction} we may assume that $\mathrm{char} \, \kk = 0$ and $I\subseteq S$ is Borel-fixed. 
Let $I= \bigoplus x^i I_i$ be the $\kk[y,z]-$decomposition and let $p = \min\big\{ i \, : \, I_i = \kk[y,z]\big\}$.
Assuming without loss of generality that $I \ne \mm^r$, 
the hypothesis $ d \leq \binom{r+2}{3}$ and the fact that $I$ is strongly stable  imply that $p < r$.

We denote by $T(I)_j$ the component of $T(I)$ of $x$-degree $j$, that is, 
$T(I)_j = \bigoplus_{\alpha_1 = j} \big| T(I)\big|_\alpha.$
A tangent vector $\xi\in T(I)_j$, 
viewed as homomorphism $\xi: I \rightarrow S/I$,
 is uniquely determined by its restrictions 
\begin{equation*}
\xi \big\vert_{x^i I_i}\, :\, x^i I_i\longrightarrow x^{i+j}\frac{\kk[y,z]}{I_{i+j}} 
\end{equation*}
where $i \geq 0$ and $0 \leq i+j < p$.
Clearly, $T(I)_j=0$ if $j \geq p$. 
On the other hand, we also have $T(I)_j=0$ if $j < -p$,
since any monomial minimal generator of $I$ has $x$-degree at most $p$ by strong stability.
For the same reason, 
it suffices to consider the restrictions  for $i \leq p$.
To summarize, every $x$-homogeneous $\xi \in T(I)$ is determined by the induced $\kk[y,z]$-linear homomorphisms
\begin{equation}\label{EqRestrictions}
\xi \big\vert_{ I_i}\, :\, I_i\longrightarrow \frac{\kk[y,z]}{I_{i+j}} 
\qquad 
\text{with}
\quad 
-p \leq j \leq p-1, 
\quad
\max(0, -j) \leq i \leq \min(p,p-j-1)
\end{equation}
where, by abuse of notation, we drop the placeholders $x^i, x^{i+j}$.

Now we can estimate the dimension of the tangent space:

\begin{align*}
\dim_\kk T(I) 
&\leq 
\sum_{j=-p}^{p-1} \sum_{i = \max(0,-j)}^{\min(p,p-j-1)} \dim_\kk\Hom\left(I_i, \frac{\kk[y,z]}{I_{i+j}}\right)
&& \text{by \eqref{EqRestrictions} } \\
&=
\sum_{j=-p}^{p-1} \sum_{i = \max(0,-j)}^{\min(p,p-j-1)} \left( \dim_\kk\frac{\kk[y,z]}{I_{i+j}}+ \dim_\kk \frac{I_i:I_{i+j}}{I_{i}}\right)
&& \text{by Proposition \ref{Propostion2RLR}}\\
&\leq
\sum_{j=-p}^{p-1} \sum_{i = \max(0,-j)}^{\min(p,p-j-1)} \left( \dim_\kk\frac{\kk[y,z]}{I_{i+j}}+ \dim_\kk \frac{\kk[y,z]}{I_{i}}\right)
&& \text{}\\
&=
\sum_{j=-p}^{-1} \sum_{i = -j}^{p} \left( \dim_\kk\frac{\kk[y,z]}{I_{i+j}}+ \dim_\kk \frac{\kk[y,z]}{I_{i}}\right)
&& \text{}\\
& +
\sum_{j=0}^{p-1} \sum_{i =0}^{p-j-1} \left( \dim_\kk\frac{\kk[y,z]}{I_{i+j}}+ \dim_\kk \frac{\kk[y,z]}{I_{i}}\right)
&& \text{}\\
&=
\sum_{i=0}^{p-1} \sum_{j = 0}^{i}  \dim_\kk\frac{\kk[y,z]}{I_{j}}+ \sum_{j=0}^{p-1} \sum_{i = p-j}^{p}\dim_\kk \frac{\kk[y,z]}{I_{i}}
&& \text{reindexing}\\
& +
\sum_{i=0}^{p-1} \sum_{j =i}^{p-1}  \dim_\kk\frac{\kk[y,z]}{I_{j}}+ \sum_{j=0}^{p-1} \sum_{i =0}^{p-j-1}\dim_\kk \frac{\kk[y,z]}{I_{i}}
&& \text{}\\
&=
(p+1) \sum_{j = 0}^{p-1}  \dim_\kk\frac{\kk[y,z]}{I_{j}}+p\sum_{i =0}^{p}\dim_\kk \frac{\kk[y,z]}{I_{i}}\\
&=
(2p+1)\dim_\kk \frac{S}{I} \leq (2r-1) {r+2 \choose 3}&&\text{by assumption} \\
& \leq \frac{4}{3}{r+2\choose 2}{r+1 \choose 2} = \frac{4}{3}\dim_\kk T(\mm^r) && \text{by Lemma \ref{LemmaFatPointPPN}.}
\end{align*}
\end{proof}

Our analysis allows  verifying Conjecture \ref{BIconj1}
 for many  monomial ideals:

\begin{cor} \label{Corollary34}
Let  $[I]\in \Hilb^d\AA^3$ be a monomial point with $d\leq {r+2 \choose 3}$.
If $x^p \in I$ with $p \leq \frac{3r+1}{4}$, then 
$\dim_\kk T(I) \leq \dim_\kk T(\mm^r)$.
\end{cor}

\begin{proof}
As in the proof of Theorem \ref{TheoremGlobalBound}, we may assume that $\mathrm{char} \, \kk = 0$ and $I\subseteq S$ is Borel-fixed: 
in fact, if $I$ is any monomial ideal and $x^p \in I $, then $x^p \in \gin\, I $ as well.
Now, 
if $p \leq \frac{3r+1}{4}$ then we can improve the estimates in the proof of Theorem \ref{TheoremGlobalBound}
obtaining 
$$
\dim_\kk T(I) \leq (2p+1) \dim_\kk\frac{S}{I} \leq \frac{6r +6}{4} {r+2\choose 3} ={r+2\choose 2}{r+1 \choose 2} =\dim_\kk T(\mm^r).
$$
\end{proof}

As observed in the proof of Theorem \ref{TheoremGlobalBound}, 
if $I$ is strongly stable and $ d = \dim_\kk (S/I)\leq {r+2 \choose 3}$ then 
$x^r \in I$.
Hence, Corollary \ref{Corollary34}
proves  Conjecture \ref{BIconj1} for
``three quarters'' of the strongly stable ideals -- in fact, often for a much larger fraction.
For example, we use this fact   in the proof of Proposition \ref{PropositionCounterexampleSturmfels},
where the search for the maximum tangent space dimension for $\Hilb^{39}\AA^3$ is reduced from all 39098 strongly stable ideals to the 2654 ones that do not  contain small powers of $x$.

Another consequence of Theorem \ref{TheoremGlobalBound} is a new  bound on the dimension of the Hilbert scheme:

\begin{cor}\label{CorollaryDimensionHS}
For $d \gg 0$ we have $\dim \Hilb^d \AA^3 \leq 3.64 \cdot  d^{\frac{4}{3}}$.
\end{cor}

\begin{proof}
Let $r \in \N$ such that $ {r +1 \choose 3} < d \leq {r+2 \choose 3}$, so  $r-1 \leq  \sqrt[3]{6d}$.
Using Theorem \ref{TheoremGlobalBound} we get
\begin{align*}
\dim \Hilb^d \AA^3 
&\leq 
\max_{I\in \Hilb^d \AA^3 } \dim_\kk T(I)
\leq \frac{4}{3}\dim_\kk T(\mm^r) = \frac{4}{3} {r+2\choose 2}{r+1 \choose 2}
 \\
&=\frac{1}{3} (r+2)(r+1)^2(r) \leq \frac{1}{3} \left(\sqrt[3]{6d}\right)^{4} + O(d) \approx  3.634\cdot  d^{\frac{4}{3}} + O(d)
\end{align*}
implying the desired asymptotic bound.
\end{proof}

\begin{remark} 
The authors in \cite{BrianconIarrobino} proved that 
$\dim \Hilb^d \AA^3 \leq 19.92\cdot d^{\frac{4}{3}}$. 
On the other hand, 
the full Conjecture \ref{BIconj1} would imply that $\dim \Hilb^d \AA^3 \leq 2.73\cdot d^{\frac{4}{3}}$ for $d \gg 0$.
\end{remark}

\section{Counterexamples to the second Brian\c{c}on-Iarrobino conjecture}\label{SectionCounterexamples}

For each $d \in \N$ let $E(d) \subseteq S$ be the unique  ideal such that $[E(d)] \in \Hilb^d\AA^3$, 
$E(d)$ is a lexsegment ideal (cf. \cite[2.4]{MillerSturmfels} or \cite[4.2]{BrunsHerzog}), and $E(d)$ is generated in at most two consecutive degrees.
In other words, $ E(d) = K +\mm^{r+1}$ for an ideal $K$ generated by a lexicographic initial segment of monomials of degree $r$, for some $r$.
It is clear that $ E(d) = \mm^r $ when $d = {r+2 \choose 3}$.
More generally, the ideal $E(d)$ behaves in many respects in the same way as the powers of $\mm$, 
attaining for every $d$ extremal number of generators, syzygies, socle monomials, and more,
see \cite{Valla,CavigliaSammartano}.
Conjecture \ref{BIconj2} stated that $E(d)$ also attains extremal tangent space dimension, for every $d$.
Sturmfels \cite{Sturmfels} disproved it by exhibiting counterexamples for $d = 8, 16$.
In this section we prove that these counterexamples are not sporadic, i.e. that Conjecture \ref{BIconj2} fails infinitely often.
Furthermore, 
we show that a weaker form of Conjecture \ref{BIconj2} proposed by Sturmfels
also has a negative answer.

We recall that the socle of $T(I)$ is easy to compute, cf. Section \ref{SectionTangentSpacePreliminaries}.
The main goal of the next results is to understand  non-socle tangent vectors.

\begin{prop}\label{PropositionTangentSpaceLd}
Let $r\in \N$ with $r\geq 3$ and 
 $d = {r+2\choose 3} + r+3$. 
Then we have
$$
\dim_\kk T(E(d)) = \left( {r+2 \choose 2} +1 \right) \left( {r+1 \choose 2} +1 \right) +7.
$$ 
\end{prop}

\begin{proof}
Denote for simplicity $E = E(d)$.
First of all, we observe that  $E = E_r + E_{r+1}$ where
$$
E_r = x^2 (x,y,z)^{r-2} + xy^2 ( y,z)^{r-3} 
\qquad
\text{and}
\qquad
E_{r+1} = (xyz^{r-1}, xz^r) + (y,z)^{r+1}
$$
and it follows that the number of generators of $E$ is 
$$
\dim_\kk \left( \frac{E}{\mm E} \right) = \dim_\kk \left( \frac{E_r}{\mm E_r} \right)+\dim_\kk \left( \frac{E_{r+1}}{\mm E_{r+1}} \right) = \left({r\choose 2} + r-2\right) + \left(2 + r + 2\right) =   {r+2 \choose 2} +1.
$$
On the other hand, we have $\soc(S/E(d)) = Z_{r-1} \oplus Z_{r}$ where
\begin{align*}
Z_{r-1} &= x^2 \cdot  \Span_\kk\left(x^{r-3},x^{r-4}y, \ldots, z^{r-3} \right) \oplus xy^2 \cdot \Span_\kk \left(y^{r-4},y^{r-5}z, \ldots, z^{r-4} \right)\\
Z_r & =  \Span_\kk \left(xyz^{r-2}, xz^{r-1},y^{r}, y^{r-1}z, \ldots, z^{r} \right)
\end{align*}
so that
$$
\dim_\kk \soc\left(\frac{S}{I}\right) = \dim_\kk (Z_{r-1})+\dim_\kk (Z_{r}) = \left({r-1\choose 2} + r-3\right) + \left(2 + r + 1\right) =   {r+1 \choose 2} +1.
$$

Using  \eqref{EquationSocleTangentSpace} we conclude that $\dim_\kk \soc(T(E)) = \big( {r+2 \choose 2} +1 \big) \big( {r+1 \choose 2} +1 \big) $.
In order to prove the desired formula for $\dim_\kk T(E)$,
it suffices to show that there are exactly 7 non-socle maps in the $\kk$-basis of $T(E)$ described in Proposition \ref{PropositionBasis}.

Fix some $\alpha\in \Z^3$ and 
a connected component $\emptyset \ne U$   of $  (\tilde{E}+{\alpha}) \setminus \, \tilde{E}$.
The associated  $S$-linear map $\varphi = \varphi_U : E \rightarrow S/E$  is
defined by $\varphi(\bfx^\beta) = \bfx^{\alpha+\beta}$ if $\alpha+\beta\in U$, $0$ otherwise,
thus $U = \widetilde{\varphi(E)} \subseteq \N^3$.
Assume that  $ \Card(U)\geq 2 $, equivalently, by Remark \ref{RemarkBasisSocle}, that $\varphi$ is not a socle map.
We will show that there are 7 possibilities for $\varphi$, in $3$ steps, and this will conclude the proof.

\underline{Step 1:} 
We claim that $\deg( \varphi( \bfx^\beta)) \geq r-1$ for all $\bfx^\beta\in I.$
\\
Define 
$
M_s = \big\{ \beta \in \widetilde{E} \, : \, \Vert \beta \Vert = s\big\}
$
for  $s\geq r$.
Assume by contradiction  $\deg( \varphi( \bfx^\beta)) = \Vert\alpha + \beta \Vert \leq r-2 $ for some minimal generator $\bfx^\beta\in E.$
In particular, $\beta \in M_r \cup M_{r+1}$ and $ \alpha+ \beta\in U \subseteq  \N^3$.

If $\beta \in M_r$ then 
 $\omega(\alpha)  = \alpha_1 + \alpha_2 + \alpha_3\leq -2$, hence
$M_r + \alpha, M_{r+1} + \alpha$ are disjoint from $\widetilde{E}$.
It follows that $M_r +\alpha \subseteq \N^3$, otherwise we can use the points of $M_r +\alpha$ and $M_{r+1} +\alpha$ to connect $\beta+\alpha$ to a point in $\Z^3\setminus \N^3$, contradicting the hypothesis that $U$ is a bounded connected component.
However, considering the points $(r, 0,0), (1, r-1, 0) \in M_r$ we deduce that $ \alpha_1 \geq -1,\alpha_2, \alpha_3 \geq 0,$ contradicting $\omega(\alpha) \leq -2$.
If $\beta \in M_{r+1}$ then 
 $\omega(\alpha) \leq -3$,
and we can apply the same argument using  
$M_{r+1} + \alpha, M_{r+2} + \alpha$
and the points $(r+1, 0,0), (0,r+1,0)$.

This verifies  the claim of Step 1. 
It follows, by the non-socle assumption on $\varphi$, that $\varphi(E) \subseteq (N_{r-1}) \subseteq S/E$ where $$
N_{r-1} = \Span_\kk\left( xyz^{r-3}, xz^{r-2}, y^{r-1},y^{r-2}z, \ldots, z^{r-1}\right)
$$
are the non-socle monomials of degree $r-1$.
The annihilators of these monomials are
\begin{equation}\label{EquationAnnihilator}
\begin{aligned}
\ann_{S/E}\left(xyz^{r-3}\right) &= \left(x,y,z^2\right)\\
\ann_{S/E}\left(\bfx^\beta\right) &= \left(x,y^2,yz,z^2\right)
\qquad
&
\text{ for }
\bfx^\beta =  xz^{r-2}, y^{r-1},\ldots, y^2z^{r-3}\\
\ann_{S/E}\left(\bfx^\beta\right) &= \mm^2
\qquad
&
\text{ for }
\bfx^\beta = yz^{r-2}, z^{r-1}.
\end{aligned}
\end{equation}
Moreover, we necessarily have $\omega(\alpha) = -2 $ or $-1$.

\underline{Step 2:} Suppose that $\omega(\alpha) = -2$. 
It follows that 
$0 \ne\varphi(\bfx^\beta) \in N_{r-1}$ for some  generator $\bfx^\beta \in E_{r+1}$;
choose   $\bfx^\beta$ to be the smallest in the lexicographic order.
By \eqref{EquationAnnihilator} we have $\varphi(\bfx^\beta z) \ne 0$.
If $x$ or $y$ divide $ \bfx^\beta$ then 
 $ \varphi(\bfx^\beta x^{-1} z)\ne 0$ or 
$  \varphi(\bfx^\beta y^{-1} z) \ne 0$, 
contradicting the choice of $\bfx^\beta$, 
hence $\bfx^\beta = z^{r+1}$. 
In particular, $\alpha_1 \in \{ 0, 1\}$.
If $\alpha_1 = 0 $,
then by \eqref{EquationAnnihilator} we get 
$y\varphi(\bfx^\gamma), z\varphi(\bfx^\gamma)\ne 0$
for all  generators $\bfx^\gamma \in E_{r+1}$ with $x \nmid \bfx^\gamma$.
We obtain
$$
 \varphi(z^{r+1}) \ne 0
\Rightarrow
 \varphi(yz^{r+1}) \ne 0
\Rightarrow
 \varphi(yz^{r}) \ne 0
\Rightarrow
\cdots
\Rightarrow
 \varphi(y^{r+1}z) \ne 0
\Rightarrow
 \varphi(y^{r+1}) \ne 0
$$
hence $\varphi$ restricts to an embedding of vector spaces 
$$
\Span_\kk\left( y^{r+1}, y^rz, \ldots, z^{r+1}\right) \hookrightarrow \Span_\kk\left( y^{r-1}, y^{r-2}z, \ldots, z^{r-1}\right)
$$
which is obviously impossible. 
Therefore  $\alpha_1 = 1$ and there are  2 possibilities for $\varphi$,
forced by the choice of $\varphi(z^{r+1}) \in \{ xyz^{r-3}, xz^{r-2} \}$.
These are the  maps $\varphi = \varphi_U$ with $\omega(\alpha) = -2 $ that are associated to the connected components
\begin{align*}
U &= \big\{ (1,1,r-3), (1,1,r-2) \big\} \subseteq \big(\widetilde{E} + (1,1,-4) \big) \setminus \widetilde{E},\\
U &= \big\{ (1,1,r-3), (1,1,r-2), (1,0,r-2), (1,0,r-1) \big\} \subseteq \big(\widetilde{E} + (1,0,-3) \big) \setminus \widetilde{E}.
\end{align*}

\underline{Step 3:} Suppose that $\omega(\alpha) = -1$. 
It follows that 
$0\ne \varphi(\bfx^\beta) \in N_{r-1}$ for some  generator $\bfx^\beta \in E_{r}$;
choose   $\bfx^\beta$ to be the smallest in the lexicographic order.
Since 
$ \varphi(\bfx^\beta z) \ne 0$  by  \eqref{EquationAnnihilator},
we have
$$
\text{if }\,
 x^3 \mid \bfx^\beta
 \Rightarrow
  \varphi(\bfx^\beta x^{-1} z)
   \ne 0,
   \quad
\text{if }\,
 x^2y \mid \bfx^\beta 
 \Rightarrow
  \varphi(\bfx^\beta y^{-1} z)
   \ne 0,
   \quad
\text{if }\,
 xy^3 \mid \bfx^\beta 
 \Rightarrow
  \varphi(\bfx^\beta y^{-1} z)
   \ne 0,
$$
yielding  either $\bfx^\beta = x^2z^{r-2}$ or $\bfx^\beta = xy^2z^{r-3}$.

If $\bfx^\gamma = x^2z^{r-2}\ne 0$, then we claim that $ x \mid \varphi( \bfx^\beta )$.
Suppose otherwise, and recall $y$ does not annihilate  monomials of $N_{r-1}$ that are not divisible by $x$.
We get
\begin{align*}
 \varphi(x^2z^{r-2}) \ne 0
& \Rightarrow
 \varphi(x^2yz^{r-2}) \ne 0
\Rightarrow
 \varphi(x^2yz^{r-3}) \ne 0
\Rightarrow
 \varphi(x^2y^2z^{r-3}) \ne 0
 \\
& \Rightarrow
 \varphi(xy^2z^{r-3}) = x^{-1}y^2z^{-1} \varphi(x^2z^{r-2}),
 \text{ contradiction.}
\end{align*}
We get 2 possibilities for $\varphi$,
forced by the choice of $\varphi(x^2z^{r-2}) \in \{ xyz^{r-3}, xz^{r-2} \}$.

If  $\bfx^\beta = xy^2z^{r-3}$, 
then we claim that $ z^{r-3} \mid \varphi(\bfx^\beta )$.
Suppose otherwise, then $x\nmid \varphi(\bfx^\beta )$.
Using again the fact that $y$ does not annihilate  monomials of $N_{r-1}$ that are not divisible by $x$
we deduce
\begin{align*}
\varphi( xy^2z^{r-3})\ne 0 
&\Rightarrow
\varphi( xy^3z^{r-3})\ne 0 
\Rightarrow
\varphi( xy^3z^{r-4})\ne 0 
\Rightarrow
\cdots
\Rightarrow
\varphi( xy^{r-1}z)\ne 0 
\\
&\Rightarrow
\varphi( xy^{r-1}) = y^{r-3}z^{-r+3} \varphi( xy^2z^{r-3}),
 \text{ contradiction.}
  \end{align*}
 On the other hand, we claim that $x \varphi(\bfx^\beta) = 0$.
 Suppose otherwise, then $\varphi( xy^2z^{r-3})\in \{ yz^{r-1}, z^{r-1}\}$ by \eqref{EquationAnnihilator},
 and in either case we reach a contradiction
\begin{align*}
x\varphi( xy^2z^{r-3})\ne 0 
&\Rightarrow
\varphi( x^2y^2z^{r-3})\ne 0 
\Rightarrow
\varphi( x^2yz^{r-3})\ne 0 
\Rightarrow
\varphi( x^2yz^{r-2})\ne 0 
\Rightarrow
\varphi( x^2z^{r-2})\ne 0 
\\
&\Rightarrow
\varphi( x^2z^{r-2}) = xy^{-2}z \varphi( xy^2z^{r-3}),
 \text{ contradiction.}
  \end{align*}
We get 3 possibilities for $\varphi$,
forced by the choice of $\varphi(xy^2z^{r-3}) \in \{ xyz^{r-3}, xz^{r-2}, y^2 z^{r-3} \}$.
In conclusion, we
have 5  maps $\varphi = \varphi_U$ with $\omega(\alpha) = -1 $,  associated to the connected components
\begin{align*}
U &= \big\{ (1,1,r-3), (1,1,r-2) \big\} \subseteq \big(\widetilde{E} + (-1,1,-1) \big) \setminus \widetilde{E},\\
U &= \big\{ (1,1,r-3), (1,1,r-2), (1,0,r-2), (1,0,r-1) \big\} \subseteq \big(\widetilde{E} + (-1,0,0) \big) \setminus \widetilde{E},\\
U &= \big\{ (1,1,r-3), (1,1,r-2) \big\} \subseteq \big(\widetilde{E} + (0,-1,0) \big) \setminus \widetilde{E},\\
U &= \big\{ (1,1,r-3), (1,1,r-2), (1,0,r-2), (1,0,r-1) \big\} \subseteq \big(\widetilde{E} + (0,-2,1) \big) \setminus \widetilde{E},\\
U &= \big\{ (0,r,0), (0,r-1,0), (0,r-1,1), \ldots, (0,3,r-4), (0,2,r-3) \big\} 
\subseteq \big(\widetilde{E} + (-1,0,0) \big) \setminus \widetilde{E}.
\end{align*}
\end{proof}

\begin{thm}\label{TheoremCounterexamples}
Let $r\in \N$ with $r\geq 3$ and 
 $d = {r+2\choose 3} + r+3$. 
The ideal 
$$
J =  x^2 (x,y,z)^{r-2} + xy( y,z)^{r-2} + (xz^r) + y(y,z)^r +(z^{r+2}) \subseteq S
$$
satisfies 
$[J] \in \Hilb^d\AA^3$ and
$
\dim_\kk T(J) > \dim_\kk T(E(d)).$
\end{thm}
\begin{proof}
The number of generators of $J$ is ${r \choose 2} + (r-1) + 1 + (r+1) +1  = {r+2 \choose 2} + 1$.
We have 
\begin{align*}
\soc\left( \frac{S}{J}\right) = \, & x^2 \cdot \Span_\kk( x^{r-3}, x^{r-2} y, \ldots, z^{r-3}) \oplus xy \cdot\Span_\kk( y^{r-3},  \ldots, z^{r-3})\\
&  \oplus \Span_\kk ( xz^{r-1}, y^r, y^{r-1}z, \ldots, yz^{r-1}, z^{r+1} ) 
\end{align*}
giving $\dim_\kk \soc(S/J ) = {  r+1 \choose 2} +1$.
By \eqref{EquationSocleTangentSpace} 
we deduce  
$
\dim_\kk  \soc\big(T(J)\big) = \big( {r+2 \choose 2} + 1 \big)\big( {r+1 \choose 2} + 1 \big).
$
In order to complete the proof,
by Proposition  \ref{PropositionTangentSpaceLd},
it suffices to exhibit at least 8 non-socle maps.

Consider the sets  $\mathcal{S} = \{x^2z^{r-2}, xy z^{r-2}, xz^r, yz^r, z^{r+2}\} \subseteq J$ and $\mathcal{T}=\{xz^{r-2}, z^{r}\} \subseteq S/J$.
Each $\bfx^\beta \in  \mathcal{S} $ has the property that $\bfx^\beta$ is the only monomial of $J$ dividing $z\bfx^\beta $ properly.
For each $\bfx^\tau \in \mathcal{T}$ we have $\ann_{S/J}(\bfx^\beta)  = (x,y,z^2)$.
It follows that for every $\bfx^\beta \in \mathcal{S}, \bfx^\tau \in \mathcal{T}$ 
there is a non-socle map $\varphi \in T(J)$ such that 
$\varphi(\bfx^\beta) = \bfx^\tau$ and 
$\varphi(\bfx^\gamma) = 0$ for every generator $\bfx^\gamma$ of $J$ with 
 $\bfx^\gamma \ne  \bfx^\beta$.
 These are the 10 basis maps $\varphi_U$ associated to the following connected components 
 $ U \subseteq \big(\widetilde{J} + \alpha \big) \setminus \widetilde{J}$
 \begin{align*}
U &= \big\{ (1,0,r-2), (1,0,r-1) \big\} 
\quad
&\text{for } \,
 \alpha = (-1,0,0), (0,-1,0), (0,0,-2), (1,-1,-2), (1,0,-4),
\\
U &= \big\{ (0,0,r), (0,0,r+1) \big\}
\quad 
&\text{for } \,
 \alpha =  (-2,0,2), (-1,-1,2), (-1,0,0), (0,-1,0), (0,0,-2).
\end{align*}
\end{proof}

\begin{remark}
For  $r=3, d= 16$, Theorem \ref{TheoremCounterexamples} recovers the counterexample of \cite[Theorem 2.3]{Sturmfels}.
\end{remark}

\begin{remark}
Using arguments similar  to Proposition \ref{PropositionTangentSpaceLd} and Theorem \ref{TheoremCounterexamples},
it is possible to prove that  for every $r \in \N, r\geq 3$ and 
$2\leq i \leq r-1$, the ideal 
$$
J =  x^2 (x,y,z)^{r-2} + xy( y,z)^{r-2} + (xz^r) + y(y,z)^r +(z^{r+i}) \subseteq S
$$
satisfies 
$[J] \in \Hilb^d\AA^3$   and
$
\dim_\kk T(J)> \dim_\kk T(E(d))$,
where $d = {r+2\choose 3} + r + i +1$.
\end{remark}

We conclude the paper by  addressing the following question of Sturmfels \cite[Problem 2.4.a]{Sturmfels}.

\begin{question}\label{QuestionSturmfels}
Is the maximum dimension of the tangent space to  $\Hilb^{d} \AA^3$ attained at an initial monomial ideal of the generic configuration of $d$ points?
\end{question}

We point out that $E(d)$ is the initial ideal of a generic configuration of $d$ points in $\AA^3$ with respect to the lexicographic order, 
as  follows e.g. from \cite[Theorem 1.2]{ConcaSidman},
so Question \ref{QuestionSturmfels} is a relaxation of Conjecture \ref{BIconj2}.
However, Question \ref{QuestionSturmfels} too has a negative answer.

\begin{prop}\label{PropositionCounterexampleSturmfels}
Assume $\text{char} \, \kk =0$. 
Question \ref{QuestionSturmfels}
has a  negative answer for $d= 39$.
\end{prop}
\begin{proof}
Let $I \subseteq S $ be the ideal of a generic configuration of 39 points in $ \AA^3$.
Suppose by contradiction that there exists a term order $<$ on $S$ such that the initial ideal $J = \mathrm{in}_<(I)$ attains the maximum tangent space dimension for 
$\Hilb^{39} \AA^3$.
Since $g\cdot I$ is also  the ideal of a generic configuration  for general $g \in \GL_3$, 
we have $J =  \mathrm{gin}_<(I)$, 
therefore $J$ is strongly stable by Lemmas \ref{LemmaGinReduction} and \ref{LemmaBorelExchangeChar0}.

We note that $x^3 \notin J$, otherwise the estimates in the proof of Theorem \ref{TheoremGlobalBound} 
yield the contradiction
$$
\dim_\kk T(J) \leq (2p+1) \dim_\kk \frac{S}{I} \leq 7\cdot 39 < 327 = \dim_\kk T\big(E(39)\big).
$$ 
Using the computer algebra system \emph{Macaulay2} \cite{Macaulay2} and the package \emph{Strongly stable ideals and Hilbert polynomials} \cite{AL},
we compute $\dim_\kk T(K)$ for all strongly stable ideals in $\Hilb^{39} \AA^3$ satisfying $x^3\notin K$
and find only one with maximum tangent space dimension, yielding 

\begin{align*}
J = \Big( & x^5, x^4y, x^4z, x^3y^2, x^3yz, x^3z^2, x^2y^3, x^2y^2z, x^2yz^2, x^2z^3, \\
				& xy^4, xy^3z, xy^2z^2, xyz^3, y^5, y^4z, y^3z^2, y^2z^3,  xz^5, yz^5, z^7\Big).
\end{align*}

Let $\overline{I} \subseteq S[t] = \kk[x,y,z,t]$ denote the homogenization of $I$, so $\overline{I}$ is the ideal of a generic configuration of 39 points in $\P^3$.
Define a term order $\prec$ on $S[t]$ by setting  $\bfx^\alpha t^a \succ \bfx^\beta t^b$ if $\bfx^\alpha > \bfx^\beta$ or $\bfx^\alpha = \bfx^\beta$ and $a>b$,
and let $H = \mathrm{in}_{\prec}\big(\overline{I}\big)$. 
As before for $J$, we have $H =\mathrm{gin}_{\prec}\big(\overline{I}\big)$, and in particular $H$ is strongly stable (with respect to the ordering of the variables $x\succ y \succ z\succ t$).
By \cite[Proposition 15.24]{Eisenbud} the saturation of $H$ is given by $\mathrm{sat}(H) = H : (x,y,z,t)^\infty = H : t^\infty$.
For any $f \in I$, we have $\mathrm{in}_\prec \big( \overline{f} \big) =  t^p\cdot \mathrm{in}_< ( {f}) $ for some $p\in \N$, and this implies that 
$J S[t] \subseteq H : t^\infty =\mathrm{sat}(H)$. 
On the other hand, both $J S[t]$ and $\mathrm{sat}(H)$ are saturated 1-dimensional homogeneous ideals of degree 39, which implies that 
$  \mathrm{sat}(H)= J S[t]$.

Computing the Hilbert function of $JS[t]$ and using the well-known formula for the Hilbert function of general points in projective space,
we determine that the vector space $  \frac{\mathrm{sat}(H)}{H}$ has dimension 1 and is concentrated in degree 5. 
In other words, $  \frac{\mathrm{sat}(H)}{H} = \Span_\kk\big( \bfx^\alpha\big)$ for some  $\bfx^\alpha\in J$ with $\deg\big( \bfx^\alpha\big)=5$.
Since  $H$ is strongly stable,
the only possibility is $\bfx^\alpha = y^2 z^3$, yielding 

\begin{align*}
H = \Big( & x^5, x^4y, x^4z, x^3y^2, x^3yz, x^3z^2, x^2y^3, x^2y^2z, x^2yz^2, x^2z^3, xy^4, \\
				&   xy^3z, xy^2z^2, xyz^3, y^5, y^4z, y^3z^2, y^2z^4, y^2z^3t,   xz^5, yz^5, z^7\Big).
\end{align*}
By \cite[Theorem 1.2]{ConcaSidman}
the ideal $H$ is a $\prec$-segment,
that is, 
whenever $\bfx^\beta \in H$ and $\bfx^\gamma \succ \bfx^\beta$ with $\deg\big( \bfx^\gamma\big) = \deg\big( \bfx^\beta\big)$,
we have $\bfx^\gamma \in H$ too.
In particular:
\begin{align*}
xyz^3 \in H, \, xy^3t \notin H &\Rightarrow  xyz^3 \succ xy^3t \Rightarrow  xyz^3 > xy^3 \Rightarrow z^3 > y^2,\\
y^2z^3t \in H, \, z^6 \notin H &\Rightarrow y^2z^3t \succ z^6 \Rightarrow y^2z^3 > z^6 \Rightarrow y^2 > z^3, 
\end{align*}
generating a contradiction and concluding the proof.
\end{proof}

\subsection*{Acknowledgements} 
The authors would like to thank David Eisenbud,  Mark Haiman, Diane Maclagan, and  the anonymous referees for several helpful suggestions.
The first author was partially supported by an NSERC PGSD scholarship.
The second author was partially supported by NSF Grant No. 1440140, while he was a Postdoctoral Fellow at the Mathematical Sciences Research Institute in Berkeley.

\end{document}